\documentclass[11pt]{article}
\usepackage[utf8]{inputenc}
\usepackage[english]{babel}
\usepackage{authblk}
\usepackage{graphicx}        
\usepackage[caption=false]{subfig}
\usepackage{multicol}        
\usepackage[bottom]{footmisc}
\usepackage{placeins}
\usepackage{changes,todonotes}
\usepackage{pgf}
\usepackage{psfrag}
\usepackage{pgfplots}
\usepackage{hyperref}
\usepackage{appendix}
\usepackage{amsmath,amssymb}
\usepackage{bm}
\usepackage{mathtools}
\usepackage{color}
\usepackage{stmaryrd}
\usepackage{amsthm}
\usepackage{blindtext}
\usepackage{caption}
\usepackage{multirow}
\usepackage[table]{xcolor}

\usepackage[margin=3cm]{geometry}

\usepackage{algorithm}
\usepackage{algorithmic}

\graphicspath{./FIGURES/}

\newcommand{\vertiii}[1]{{\left\vert\kern-0.25ex\left\vert\kern-0.25ex\left\vert #1 
    \right\vert\kern-0.25ex\right\vert\kern-0.25ex\right\vert}}

\newcommand{\trinorm}[1]{{\vert\kern-0.25ex\vert\kern-0.25ex\vert #1 \vert\kern-0.25ex\vert\kern-0.25ex\vert}}

\definecolor{farbe}{gray}{0.80}

\newtheorem{theorem}{Theorem}

\newtheorem{definition}{Definition}

\newtheorem{corollary}{Corollary}

\newtheorem{remark}{Remark}

\pgfplotsset{compat=1.18}

\begin{document}

\title{A Scalable Deflated Conjugate Gradient Solver for the Time-Dependent Pseudo-Stress Stokes Problem}

\author[$\star$]{Alessandra Cancrini}
\author[$\star$]{Gabriele Ciaramella}
\author[$\star$]{Paola F. Antonietti}

\affil[$\star$]{MOX, Laboratory for Modeling and Scientific Computing, Dipartimento di Matematica, Politecnico di Milano, Piazza Leonardo da Vinci 32, I-20133 Milano, Italy}

\affil[ ]{\texttt {\{alessandra.cancrini,gabriele.ciaramella,paola.antonietti\}@polimi.it}}

\maketitle

\begin{abstract}
We propose a novel iterative solution framework for the unsteady Stokes equations in the pseudo-stress formulation. When solving this class of problems by using implicit time-integration schemes, standard solvers suffer from deteriorating convergence properties for small time steps, independently of the chosen space discretisation method. This is due to the singular modes of the dev-dev operator. For this reason, we introduce a computational framework obtained by combining a deflated Conjugate Gradient method with a W-cycle multigrid scheme that employs a Restricted Additive Schwarz smoother. The key point is to choose the deflation subspace so that the inner system to be solved within a deflated Conjugate Gradient scheme corresponds to a Laplace problem defined on the singular modes of the original dev-dev operator. This results to be independent of the spatial discretisation method and allows one to use efficient multigrid iterative solvers. Numerical experiments show that the proposed strategy significantly accelerates the Conjugate Gradient convergence and provides stable performance with respect to the time step, confirming its robustness for solving linear systems in the pseudo-stress framework.

\end{abstract}



\section{Introduction}

The numerical approximation of incompressible viscous flows is a central problem in computational fluid dynamics. In this context, the stress–velocity–pressure formulation of the unsteady Stokes equations has gained considerable interest, particularly due to its relevance for non-Newtonian fluid flow models (see e.g. \cite{Baaijens1998,Keunings2001}) and coupled interface problems, where an accurate representation of the stress tensor plays a crucial role~\cite{Gerritsma1999, Qiu.Zhao2024}. Although the stress can be reconstructed a posteriori in the velocity–pressure formulation by differentiating the velocity field, this approach might compromise the overall accuracy. It is also worth noting that the stress–velocity–pressure formulation introduces an additional difficulty, namely the enforcement of the stress tensor’s symmetry at the discrete level \cite{Arnold1984,Arnold2002}.
A possible strategy to circumvent this issue is to introduce the concept of pseudo-stress, which is a nonsymmetric tensor \cite{CaiLee2004, Cai2007}. Motivated by these features, different works have considered the unsteady Stokes problem written in terms of the pseudo-stress variable only, employing DG methods on polytopal grids (see \cite{Cancrini,TesiCancrini} and references therein).

The present work focuses on the development of a novel computational framework for the solution of discrete incompressible viscous flows in pseudo-stress form.
The final goal is to develop a scalable solver that is robust with respect to both the time step $\Delta t$ and the mesh size $h$. As a first step, in this work we focus on robustness with respect to $\Delta t$ and investigate how the performance of classical iterative methods is affected by its variation. Numerical experiments show that the Conjugate Gradient (CG) method  \cite{Magnus, Saad2003} suffers from a severe deterioration of convergence as $\Delta t$ decreases. This behavior is due to the conditioning of the system matrix \cite{Saad2003, vanderSluis1986}, which scales as $1/\Delta t$. Therefore, smaller time steps produce increasingly ill-conditioned linear systems, and thus require a larger number of iterations to reduce the relative residual below a prescribed tolerance \cite{CancriniCiaramella}.
It is well known that, when the condition number becomes large, improving the spectral properties of the system matrix is  often essential to achieve scalability.
This issue arises independently of the spatial discretisation employed. Indeed, when implicit time integration schemes, such as $\theta$-methods or Backward Differentiation Formulae (BDF), are applied to the pseudo-stress formulation of the unsteady Stokes equations, the resulting algebraic problem is symmetric and positive definite. Nevertheless, already at the semi-discrete level, the spectrum of the operator exhibits a dependence on $1/\Delta t$, leading to increasingly ill-conditioned systems as the time step decreases. As a consequence, the convergence of Krylov subspace methods deteriorates significantly for small values of $\Delta t$, regardless of the finite-dimensional approximation adopted in space. Similar behavior has been observed and investigated in our previous work \cite{CancriniCiaramella}.

To address this issue, we propose a novel solution framework based on a deflated Conjugate Gradient method \cite{Dostl1988ConjugateGM, Kahl, Nicolaides} (see also \cite{Fank2001, Gaul2013, Saad2000}). Deflation techniques are designed to accelerate Krylov subspace iterations by removing from the spectrum those components responsible for slow convergence. A central aspect of the deflated CG method is the choice of the deflation subspace $\mathbf{S}\subset\mathbb{R}^n$, since this choice fully determines the effectiveness of the deflation. For the problem under consideration, the dependence on $1/\Delta t$ can be removed by choosing a subspace $\mathbf{S}$ that contains all singular modes of the deviatoric–deviatoric operator.
Once $\mathbf{S}$ is fixed, the resulting inner operator, namely the projection of the original problem to $\mathbf{S}$, must be efficiently applied to a vector, in order to preserve the overall performance of the iterative scheme. This application requires the solution of a system that arises from the restriction of the original problem to the subspace $\mathbf{S}$. We will see that, for a suitably chosen $\mathbf{S}$, which contains singular modes of the deviatoric–deviatoric operator, this inner system coincides with a Laplace equation, opening the possibility of employing fast and robust iterative solvers.
An efficient implementation of the deflation strategy also requires the computation of suitable projection operators. These projections involve the solution of the original problem restricted to the deflation subspace $\mathbf{S}$ and therefore represent a crucial component of the overall computational cost. One of the main theoretical contributions of this work is to show that such a restricted operator is equivalent to a Laplace problem both at the continuous and at the discrete level. Furthermore, in the discrete setting, the restricted problem inherits exactly the spatial discretisation adopted for the original pseudo-stress formulation. Therefore, if a continuous Galerkin (CG) or a discontinuous Galerkin (DG) discretisation is employed for the original problem, then the corresponding restricted operator becomes a Laplace problem discretised by the same CG or DG method. This property makes the proposed framework completely independent of the chosen spatial approximation and therefore applicable to a broad class of discretisation techniques.

The identification of the restricted problem with a Laplace operator is particularly advantageous, since efficient multigrid solvers for elliptic problems are well established in the literature \cite{Bramble1996TheAO, Bramble1992, Bramble1991TheAO} (see \cite{AntoniettiSarti2014, Antonietti2017VcycleMA, Antonietti2017AUA, Gopalakrishnan} for multigrid applications related to DG methods). For this reason, the proposed framework combines an outer deflated CG iteration with an inner multigrid solver dedicated to the restricted problem. As smoother for the multigrid scheme, we employ a Restricted Additive Schwarz (RAS) iteration \cite{Cai1999ARA,Ciaramella,Gander2008SchwarzMO}, which is well suited to block-structured operators and meshes of general topology. The combination of a W-cycle multigrid strategy with a RAS smoother yields a robust and scalable iterative scheme for the inner system, improving the overall performance of the deflated CG approach.
To further reduce the computational cost, the inner multigrid solver is not required to solve the restricted problem to a fixed accuracy. Instead, its stopping criterion is dynamically adjusted according to the residual of the outer CG iteration \cite{Kahl}. This strategy substantially decreases the total computational effort while preserving the convergence properties of the overall algorithm. Since the inner system is solved only approximately, its influence on the outer iteration may vary between steps. To prevent possible degradation of convergence behavior caused by these variations, the deflated CG approach can be enhanced with a flexible variant of the CG method \cite{axelsson1996iterative,Notay,Saad2003}, which provides greater stability when the system operator changes during iterations.
The numerical experiments presented in this work demonstrate that the proposed methodology significantly improves the convergence behavior of classical Krylov solvers and exhibits robustness with respect to the time step. These results confirm the effectiveness of the proposed deflation strategy and highlight its potential as a building block towards fully scalable solvers for the pseudo-stress formulation of the unsteady Stokes equations.

The structure of the paper is as follows. Section~\ref{sec:stokes_problem} presents the pseudo-stress weak formulation of the unsteady Stokes problem, the associated time semi-discrete formulation, and the reduced system arising from it. Section~\ref{sec:fully_discrete} is devoted to the fully-discrete setting: we introduce a general spatial discretisation, derive the corresponding discrete reduced system, and present examples of spatial discretisation schemes showing how the abstract theory can be applied in practice. In Section~\ref{sec:solvers} we develop a deflated CG strategy to stabilize the iterative solution of the fully discrete system as $\Delta t \to 0$, and we propose a multigrid approach with RAS smoothing to efficiently solve the inner system arising from the deflation procedure. 
We also present a flexible CG variant for solving the system, which improves convergence when the operator changes during iterations. Numerical tests are reported in Section~\ref{sec:num_results}, where we assess the performance of the proposed approach by means of experiments in both two- and three-dimensional settings, employing the spatial discretisation based on a discontinuous Galerkin method on polytopal meshes.
Finally, in Section~\ref{sec:conclusions} we draw some conclusions and an outlook of possible extension of this work.

\section{Variational formulation and time discretisation}\label{sec:stokes_problem}
This section is devoted to the variational setting of the problem. We first introduce the notation and the pseudo-stress weak formulation, then derive a time semi-discrete formulation, and finally present the reduced system together with its Laplacian structure.

\subsection{Notation}
Let $\Omega\subset\mathbb{R}^d$, \textcolor{black}{$d = 2, 3$}, be a \textcolor{black}{bounded open} domain with Lipschitz boundary $\partial\Omega$.
In what follows, for $X \subseteq \Omega$, the notation $\bm{L}^2(X)$ is adopted in place of $[L^2(X)]^d$, and $\mathbb{L}^2(X)$ in place of $[L^2(X)]^{d\times d}$.
The scalar product in $L^2(X)$ is denoted by $(\cdot,\cdot)_X$, with the associated norm $\| \cdot \|_X$.  
Similarly, the Sobolev spaces $\bm{H}^\ell(X)$ are defined as $[H^\ell(X)]^d$, with $\ell\geq 0$, equipped with the norm $\| \cdot \|_{\ell,X}$, assuming conventionally that $\bm{H}^0(X)\equiv\bm{L}^2(X)$. 
In addition, we will use $\bm{H}(\textrm{div},X)$ to denote the space of $\bm{L}^2(X)$ vector fields with square-integrable divergence. Then, the notation $\mathbb{H}(\textrm{div},X)$ is used for the space of tensor fields with rows belonging to $\bm{H}(\textrm{div},X)$.
Finally, \textcolor{black}{for $\Gamma\subseteq\partial\Omega$ closed subset of $\partial\Omega$}, we consider the space $H^{\frac12}(\Gamma) = \{ v \in L^2(\Gamma)\, |\, \exists u\in H^1(\Omega): u_{|\Gamma} = v \}$ and its dual space $H^{-\frac{1}{2}}(\Gamma)$. 
The duality product between $v\in H^{\frac12}(\Gamma)$ and $w\in H^{-\frac12}(\Gamma)$ is denoted by $\langle v,w \rangle_{\Gamma}$.

\subsection{Pseudo-stress weak formulation}

We focus on the time-dependent Stokes problem, which models incompressible viscous free flows, and formulate it in terms of a pseudo-stress unknown rather than its classical form. The derivation of the pseudo-stress formulation for the unsteady Stokes problem is presented in \cite{Cancrini}. The pseudo-stress variable is defined as $\boldsymbol{\sigma}(\boldsymbol{u},p) = \mu\boldsymbol{\nabla}\boldsymbol{u} - p\mathbb{I}_d$, where $\boldsymbol{u}$ is the flow velocity, $p$ its pressure, and $\mathbb{I}_d$ the identity matrix in $\mathbb{R}^{d \times d}$. Let ${\rm \textbf{dev}}(\boldsymbol{\tau}) = \boldsymbol{\tau} - \frac1d {\rm tr}(\boldsymbol{\tau})\mathbb{I}_d$ be the deviatoric operator, where ${\rm tr}(\cdot)$ is the trace operator. 
The boundary of $\Omega$ is partitioned as $ \Gamma_D\cup \Gamma_N = \partial\Omega$, with $\Gamma_D\cap \Gamma_N = \emptyset$. For simplicity, we assume both $\left|  \Gamma_D \right| > 0$ and $\left|   \Gamma_N \right| > 0$, with $|\cdot|$ denoting the Hausdorff measure.
To strongly enforce the essential traction condition on $\Gamma_N$, we define the subspace 
$$
\bm V = \mathbb{H}_{0,\Gamma_N}({\rm div}, \Omega) = \{\boldsymbol{\eta}\in\mathbb{H}({\rm div}, \Omega)\ | \ \langle\bm\eta\ \bm n,\bm v\rangle_{\partial\Omega}=0 \ \ \forall \bm v\in \bm H^1_{0,\Gamma_D}(\Omega)\},
$$
where $\bm H^1_{0,\Gamma_D}(\Omega)= \{\boldsymbol{v} \in \bm H^1(\Omega)^d \;|\; \boldsymbol{v}=\boldsymbol{0} \; \text{on}\; \Gamma_D\}$.
Then, the corresponding weak formulation reads as: for any $t\in (0, T]$, find $\bm \sigma(t) \in \mathbb{H}_{0,\Gamma_N}({\rm div}, \Omega)$ such that  
\begin{equation}\label{weak_pb}
\mathcal{M}(\partial_t \bm \sigma, \bm \tau) + \mathcal{A} (\bm \sigma, \bm \tau) = F(\bm \tau) \quad \quad \quad \forall \, \boldsymbol{\tau}\in \mathbb{H}_{0,\Gamma_N}({\rm div}, \Omega), 
\end{equation}
where for any $\bm \sigma, \bm \tau \in \mathbb{H}_{0,\Gamma_N}({\rm div}, \Omega)$ we have defined
\begin{eqnarray*}
    \mathcal{M}(\bm \sigma, \bm \tau) & = & ( \mu^{-1} \boldsymbol{{\rm dev}}(\boldsymbol{\sigma}), \boldsymbol{{\rm dev}}(\boldsymbol{\tau}) )_{\Omega}, \quad  
    \mathcal{A}(\bm \sigma, \bm \tau) = (\boldsymbol{\nabla}\cdot\boldsymbol{\sigma}, \boldsymbol{\nabla}\cdot\boldsymbol{\tau})_{\Omega}, \\
    F(\bm \tau) & =  &(\boldsymbol{F}, \boldsymbol{\tau})_{\Omega}  +  \langle\bm g_D, \boldsymbol{\tau}\,\boldsymbol{n}\rangle\textcolor{black}{_{\partial \Omega}},
\end{eqnarray*}
with $\mu>0$ the fluid viscosity, and $T>0$ the final observation time. We assume that $\boldsymbol{F}$ is sufficiently regular, 
$\boldsymbol{g}_D(t)\in \bm H^{\frac12}(\partial \Omega)$, and that the initial condition ${\rm \textbf{dev}}(\boldsymbol{\sigma})(\cdot,t=0) = \boldsymbol{\sigma}_0 \in \mathbb{L}^2(\Omega)$. 
We point out that, owing to the definition of the stress space $\mathbb{H}_{0,\Gamma_N}({\rm div}, \Omega)$, the duality pairing on the right-hand side of (\ref{weak_pb}) is taken over the whole boundary $\partial \Omega$, and not only over $\Gamma_D$.
To ensure that this term is well defined, following \cite{Cancrini}, we assume that, for all $t$, the Dirichlet datum $\boldsymbol{g}_D=\boldsymbol{g}_D(t)$ is the trace of a function in $\boldsymbol{H}^1(\Omega)$, which, with a slight abuse of notation we still denote by $\boldsymbol{g}_D$, so that $\boldsymbol{g}_D(t)\in \bm H^{\frac12}(\partial \Omega)$. This allows us to write the duality over $\partial\Omega$ in (\ref{weak_pb}). For further details, the reader is referred to \cite{Cancrini}.

\subsection{Time semi-discrete formulation}
We now introduce a generic semi-discrete formulation of \eqref{weak_pb}, resulting from a discretisation in time.
We restrict our attention to two families of implicit
methods: the $\theta$-method and the BDF schemes.
Let $\alpha = \alpha(\Delta t)$ denote the coefficient associated with the
chosen time discretisation (e.g. $\alpha=\Delta t$ for implicit Euler, 
$\alpha=\Delta t/2$ for Crank--Nicolson). Then, given the initial datum $\bm \sigma^0$, the semi-discrete problem reads: for any $n=1,\dots, N_T$ find $\bm \sigma^{n}\in \mathbb{H}_{0,\Gamma_N}({\rm div}, \Omega)$ such that
\begin{equation}\label{weak_IE}
\mathcal{M}( \bm \sigma^{n}, \bm \tau) + \alpha  \mathcal{A} (\bm \sigma^{n}, \bm \tau) = F_n^{*}(\bm \tau) \quad \quad \quad \forall \, \boldsymbol{\tau}\in \mathbb{H}_{0,\Gamma_N}({\rm div}, \Omega), 
\end{equation}
where superscript $n$ means the approximation/evaluation of the given quantity at time $t_n = n \Delta t$, $n=1,\dots, N_T$, while $F_n^{*}(\cdot)$ denotes the load functional  associated with the chosen implicit time discretisation. In particular, $F_n^{*}(\cdot)$ depends on previous time steps: only on
$n-1$ in a $\theta$-method, and on multiple past steps in BDF schemes. 

Let us rewrite
the forms $\mathcal{M}(\cdot,\cdot), \ \mathcal{A}(\cdot,\cdot)$ and $F_n^*(\cdot)$ in a componentwise form: 
\begin{equation}\label{component_form}
\begin{aligned}
    & \mathcal{M}(\bm \sigma, \bm \tau) =  \mu^{-1} \Bigg[ \sum_{k=1}^d \Big(1-\tfrac{1}{d}\Big) \ m(\sigma_{kk}, \tau_{kk}) + \sum_{\substack{k, i=1 \\ k\neq i}}^d \Big( -\tfrac{1}{d} \ m(\sigma_{kk}, \tau_{ii}) + m(\sigma_{ki}, \tau_{ki}) \Big) \Bigg], 
\\ & 
    \mathcal{A}(\bm \sigma, \bm \tau) =  \sum_{k, i, j=1}^d a_{ij}(\sigma_{ki}, \tau_{kj}), \qquad \qquad  F_n^*(\bm \tau) =  \sum_{i, j=1}^d f_{ij}(\tau_{ij}). 
\end{aligned}
\end{equation}
In this way, the formulation is kept at an abstract level, and the associated bilinear and linear forms $m(\cdot,\cdot), \ a_{ij}(\cdot,\cdot)$ and $f_{ij}(\cdot)$ 
are defined at the continuous spatial level as follows:
\begin{align*}
    m(\sigma, \tau) = \int_\Omega \sigma     \ \tau \ \text{d}x, \quad \quad a_{ij}(\sigma, \tau) = \int_\Omega \partial_i \sigma \ \partial_j \tau \ \text{d}x, \quad \quad f_{ij}(\tau) = \int_\Omega F^*_{ij} \  \tau \  \text{d}x.
\end{align*}

\subsection{Reduced system and Laplacian structure}\label{reduced_continuous}

Since the bilinear form $\mathcal{M}(\cdot,\cdot)$ has a nontrivial kernel,
there exist nonzero functions in $\bm V$ that are not detected by 
$\mathcal{M}(\cdot,\cdot)$. To properly treat these singular modes, we introduce a \textit{reduced system} obtained by restricting (\ref{weak_IE}) to the kernel of $\mathcal{M}(\cdot,\cdot)$ defined as
\begin{equation}\label{spaceS}
\begin{aligned}
\mathbf{S} = \ker(\mathcal{M}) 
&= \Big\{\, \boldsymbol{\sigma} \in \boldsymbol{V} \; | \ \;
\mathcal{M}(\boldsymbol{\sigma},\boldsymbol{\tau}) = 0
\quad \  \forall\, \boldsymbol{\tau} \in \boldsymbol{V} \,\Big\} \\
&=\{ \bm \sigma \in \bm {V} \ | \  \sigma_{ii} = \sigma_{jj} \ \ \forall i,j = 1,\dots,d,
        \ \ \text{and } \ \sigma_{ij} = 0 \ \text{ for } i \neq j  \}.
\end{aligned}
\end{equation}
Both bilinear forms $\mathcal{M}(\cdot,\cdot)$ and $\mathcal{A}(\cdot,\cdot)$ induce seminorms on the underlying space and are therefore only positive semi-definite, i.e., they each admit a nontrivial kernel.  As shown in~\cite{Cancrini}, the kernels of the two forms intersect only trivially, i.e., $
\ker(\mathcal{M}) \cap \ker(\mathcal{A}) = \{0\},$
\textcolor{black}{which guarantees that the combined operator is positive definite, and hence induces a norm with respect to which coercivity holds \cite{Cancrini}.}
On the subspace $\mathbf{S}$, the contribution of $\mathcal{M}(\cdot,\cdot)$ vanishes, while $\mathcal{A}(\cdot,\cdot)$ becomes positive definite. On the complementary subspace, $\mathcal{M}(\cdot,\cdot)$ is positive definite whereas $\mathcal{A}(\cdot,\cdot)$ remains positive semi-definite. The same property persists at the discrete level. \textcolor{black}{This structure also explains the $\mathcal{O}(1/\Delta t)$ growth of the 
condition number of the combined operator: the singular modes of 
$\mathcal{M}(\cdot,\cdot)$ force the smallest eigenvalue to scale as 
$\mathcal{O}(\Delta t)$, while the largest remains $\mathcal{O}(1)$. Hence, the condition number of the corresponding system matrix grows as $\mathcal{O}(1/\Delta t)$.}

\begin{definition}[Continuous Reduced System]\label{reduced_cont}
The reduced system, obtained by restricting the time semi-discrete problem (\ref{weak_IE}) to the subspace $\mathbf{S}$ defined in (\ref{spaceS}), reads as follows:
given $\bm \sigma^0$, for any $n=1,\dots, N_T$ find $\bm \sigma^{n}\in \mathbf{S}$ such that
\begin{equation*}
\mathcal{M}( \bm \sigma^{n}, \bm \tau) + \alpha  \mathcal{A} (\bm \sigma^{n}, \bm \tau) = F_n^{*}(\bm \tau) \quad \quad \quad \forall \, \boldsymbol{\tau}\in \mathbf{S}.
\end{equation*}
\end{definition}
Notice that, for any $\bm \sigma \in \mathbf{S} \subset \mathbb{H}(\textrm{div},\Omega)$, the corresponding $i$th row of $\bm \sigma$ has only one nonzero component that is $\sigma_{ii}$.
Thus, recalling that 
$\mathbb{H}(\textrm{div},\Omega)$ is the space of tensor fields with rows belong to $\bm{H}(\textrm{div},\Omega)$, the diagonal components $\sigma_{ii}=\sigma_{jj}=\sigma$ (for all $i$ and $j$) belong to $H^1(\Omega)$ and satisfy $\langle\sigma\ \bm n,\bm v\rangle_{\partial\Omega}=0 \ \ \forall \bm v\in \bm H^1_{0,\Gamma_D}(\Omega)$. Hence, $\sigma_{ii}=\sigma_{jj}=\sigma \in H^1_{0,\Gamma_N}(\Omega)$. 

As a result, the reduced system is a Laplace problem, as stated in the following theorem.

\begin{theorem}[Laplacian Structure of the Continuous Reduced System] \label{thm:thm1}
Let (\ref{weak_IE}) be the semi-discrete formulation of the pseudo-stress Stokes problem in dimension $d$.
Then, the reduced system in Definition~\ref{reduced_cont} corresponds to a
Laplace problem, i.e., it takes the form: find $ \bm \sigma \in \mathbf{S}$ (equivalently $\sigma \in  H^1_{0,\Gamma_N}(\Omega)$) such that
\begin{eqnarray*}
     \alpha \sum_{i=1}^d a_{ii}(\sigma, \tau) =  \sum_{i=1}^d f_{ii}(\tau) \quad \quad \quad \forall \, \tau \in  H^1_{0,\Gamma_N}(\Omega).
\end{eqnarray*}
\end{theorem}
\begin{proof}
Recalling the representations (\ref{component_form}), the reduced system reads as follows: find $\sigma \in H^1_{0,\Gamma_N}(\Omega)$ such that
\begin{eqnarray*}
     \mu^{-1} \Big[ d \  \Big(1-\tfrac{1}{d}\Big) \ m( \sigma,  \tau) -  \tfrac{d(d-1)}{d} \ m( \sigma,  \tau)  \Big] + \alpha \sum_{i=1}^d a_{ii}( \sigma,  \tau) =  \sum_{i=1}^d f_{ii}( \tau) \ \ \ \forall \tau \in  H^1_{0,\Gamma_N}(\Omega),
\end{eqnarray*}
so we have $\alpha \sum_{i=1}^d a_{ii}( \sigma, \tau) =  \sum_{i=1}^d f_{ii}( \tau)$.
Hence, the reduced system corresponds to the standard variational formulation of a Laplace problem.
\end{proof}

\section{Fully-discrete formulation} \label{sec:fully_discrete}
We can now introduce a fully-discrete formulation by using a generic  spatial discretisation: let $\bm {V}_h$ be a suitable finite-dimensional space
approximating $\bm {V}$.  
The discrete space $\bm {V}_h$ is characterized by a set of local degrees of
freedom associated with mesh elements or control volumes, and by a polynomial
degree $p_\kappa \ge 1$ (possibly varying from element to element). Then, the fully-discrete approximation to \eqref{weak_pb} reads as follows: for any $n=1,\dots, N_T$ find $\bm \sigma_h^n \in \bm V_h$ such that
\begin{equation}\label{weak_dg}
   \begin{cases}
\mathcal{M}^h( \bm \sigma^{n}_h, \bm \tau_h) + \alpha \mathcal{A}^h (\bm \sigma^{n}_h, \bm \tau_h) =  F_n^{*,h}(\bm \tau_h) & \forall \, \bm \tau_h \in  \bm V_h,  \\
(\bm \sigma_h^0, \bm \tau_h) = (\bm \sigma_{0,h}, \bm \tau_h)  & \forall \, \bm \tau_h \in  \bm V_h,
   \end{cases}
\end{equation}
where $\bm \sigma_{0,h}$ is the  $\bm L^2$-projection of $\bm \sigma_0$ onto  $\bm V_h$. To obtain the corresponding algebraic formulation, we introduce a basis $\{ \phi_i,\  i=1, \dots ,N_h \}$, with $N_h = \dim(\bm V_h)$, for the space $\bm V_h$ and express $\boldsymbol{\sigma}_h$ as a linear combination of the basis functions $ \phi_i$.
We denote by $M$ (resp. $A$) the matrix representation of the bilinear form  $\mathcal{M}^h(\cdot,\cdot)$ (resp. $\mathcal{A}^h(\cdot, \cdot)$), and by $\bm f^{*}$ the vector representation of the linear functional $ F_n^{*,h}(\cdot)$. \textcolor{black}{Although $\bm f^{*}$ depends on $n$, we will focus on the solution of the linear
system at a fixed time step and therefore omit this dependence.} The algebraic formulation of \eqref{weak_dg} reduces to: given $\bm \sigma^0_h$, for any $n=1,\dots, N_T$ find $\bm \sigma_h^{n}$ such that
\begin{equation} \label{pb_study}
A^*  \bm \sigma_h^{n} =  \bm f^{*}.
\end{equation}
Here, we have introduced $A^* = M + \alpha A$, which corresponds to $\mathcal{A}^{*,h}(\cdot, \cdot) = \mathcal{M}^h(\cdot, \cdot) + \alpha \mathcal{A}^h(\cdot, \cdot)$, defining a coercive bilinear form. 

\subsection{Discrete reduced system}\label{reduced_discrete}
After establishing the Laplacian structure of the continuous reduced system in Section~\ref{reduced_continuous}, 
we now show that an analogous result holds for the fully-discrete formulation. In particular, in the discrete setting, the forms $\mathcal{M}(\cdot,\cdot)$, $\mathcal{A}(\cdot,\cdot)$ and $F_n^*(\cdot)$ are replaced by their discrete counterparts $\mathcal{M}^h(\cdot,\cdot)$, $\mathcal{A}^h(\cdot,\cdot)$ and $F_n^{*,h}(\cdot)$. Accordingly, we define the associated discrete bilinear and linear forms by $m^h(\cdot,\cdot)$, $a_{ij}^h(\cdot,\cdot)$, and $f_{ij}^h(\cdot)$,  which depend on the chosen spatial discretisation method.
\begin{definition}[Discrete Reduced System]\label{reduced_discr}
The discrete reduced system, obtained by restricting the fully-discrete problem (\ref{weak_dg}) to the subspace $\mathbf{S}^h = \ker(\mathcal{M}^h)$, reads as follows:
given $\bm \sigma^0_h$, for any $n=1,\dots, N_T$ find $\bm \sigma^{n}_h\in \mathbf{S}^h$ such that
\begin{equation*}
\mathcal{M}^h( \bm \sigma^{n}_h, \bm \tau_h) + \alpha  \mathcal{A}^h (\bm \sigma^{n}_h, \bm \tau_h) = F_n^{*,h}(\bm \tau_h) \quad \quad \quad \forall \, \boldsymbol{\tau}_h\in \mathbf{S}^h.
\end{equation*}
\end{definition}
The next theorem is the discrete counterpart of Theorem \ref{thm:thm1}.

\begin{theorem}[Laplacian Structure of the Discrete Reduced System]
\label{thm:thm2}
Let (\ref{weak_dg}) be 
in dimension $d$. Then, the discrete reduced system of 
Definition~\ref{reduced_discr} corresponds to a discrete Laplacian problem:
$$
    \alpha \sum_{i=1}^d a_{ii}^h(\sigma_h, \tau_h)
    = \sum_{i=1}^d f_{ii}^h(\tau_h),
$$
where $a_{ii}^h(\cdot,\cdot)$ and $f_{ii}^h(\cdot)$ denote the discrete 
bilinear and linear forms associated with the chosen spatial 
discretisation method.
\end{theorem}
\begin{proof}
    The proof follows directly from that of Theorem~\ref{thm:thm1}, with adaptations to the discrete setting.
\end{proof}

We now consider the algebraic formulation corresponding to Theorem~\ref{thm:thm2}. Let $V\in \mathbb{R}^{N_h \times m}$ be a matrix whose columns form a basis for $\mathbf{S}^h$. Then, for a generic spatial dimension $d$, the kernel of the matrix $M$ is the span of the vectors obtained by a vectorization (by row) of the matrices $\bm \sigma_h$ having an identity block in correspondence
with each diagonal
component $\sigma_{ii}$ of the pseudo-stress tensor, and zero blocks elsewhere. Given the subset $\mathcal{I}_d = \{ j \in \mathbb{N} \ | \ j = (i-1)d + i \ \  \text{ for} \ i = 1,\dots,d \}$, the identity matrix 
$I_d \in \mathbb{R}^{N_h/d^2\times N_h/d^2}$, and a vector $\bm v \in \mathbb{R}^{d^2}$, we define 
\begin{equation} \label{Vmatrix}
\textcolor{black}{V = 
\frac{1}{\sqrt{d}}\,
(\bm v \otimes I_d), \quad \mathrm{span}(V) = \ker(M),} \quad \text{with} \quad v_j =
\begin{cases}
1 & \text{if } j \in \mathcal{I}_d, \\
0 & \text{otherwise},
\end{cases}
\end{equation}
\textcolor{black}{where $\mathrm{span}(V)$ is understood as the span of the columns of the matrix $V$.}
 In order to characterise the algebraic structure of the matrix that represents the reduced system, i.e. $V^\top A^* V$, we begin by illustrating the structure of the algebraic problem~(\ref{pb_study}) in the cases $d=2$ and $d=3$. 
 
 For $d=2$, the pseudo-stress tensor $\bm \sigma \in \mathbb{R}^{2 \times 2}$ is represented by its four components, which we collect into the vector of unknowns $
{\bm \sigma}=
\begin{bmatrix}
{\sigma}_{11} \ \ 
{\sigma}_{12} \ \  
{\sigma}_{21} \ \ 
{\sigma}_{22}
\end{bmatrix}^\top.$
Then, the matrices $M$ and $A$ become
{\setlength{\arraycolsep}{3pt}
 \renewcommand{\arraystretch}{0.7} \begin{align*} \label{matrices2D}
M = \frac{1}{2} \begin{bmatrix}
    1 & 0 & 0 & -1 \\
    0 & 2 & 0 & 0 \\
    0 & 0 & 2 & 0 \\
    -1 & 0 & 0 & 1 \\
\end{bmatrix} \otimes M_1, \quad 
A = \begin{bmatrix}
    C_2 &  0 \\
     0 & C_2 \\
\end{bmatrix} \quad \text{where} \ \ 
C_2 = \begin{bmatrix}
    B_{11} & B_{12} \\[0.5em]
    B_{12}^\top & B_{22} \\
\end{bmatrix},
\end{align*}}
where $M_1, B_{ij} \in \mathbb{R}^{N_h/4\times N_h/4}, \text{ for } i,j \in \{1,2\}$, are defined according to the chosen numerical discretisation method.
In particular, the blocks $B_{ij}$ are directly derived from the components of the discrete form $a^h_{ij}(\cdot,\cdot)$ for $i,j \in \{1,2\}$. 

We now study the case $d=3$: the pseudo-stress tensor $\bm{\sigma} \in \mathbb{R}^{3 \times 3}$, and the vector of unknowns is assembled by collecting all nine components into a single column vector $
\bm{\sigma} =
\begin{bmatrix}
\sigma_{11}  \ \  
\sigma_{12} \ \   
\sigma_{13} \ \  
\sigma_{21} \ \  
\sigma_{22} \ \  
\sigma_{23} \ \  
\sigma_{31} \ \  
\sigma_{32} \ \  
\sigma_{33}
\end{bmatrix}^\top. $
For $d=3$, the matrices $M$ and $A$ become
{\setlength{\arraycolsep}{3pt}
 \renewcommand{\arraystretch}{0.5} \begin{eqnarray*}
M  = \frac{1}{3}
\begin{bmatrix} 
2 & 0 & 0 & 0 & -1 & 0 & 0 & 0 & -1 \\
0 & 3 & 0 & 0 & 0 & 0 & 0 & 0 & 0 \\
0 & 0 & 3 & 0 & 0 & 0 & 0 & 0 & 0 \\
0 & 0 & 0 & 3 & 0 & 0 & 0 & 0 & 0 \\
-1 & 0 & 0 & 0 & 2 & 0 & 0 & 0 & -1 \\
0 & 0 & 0 & 0 & 0 & 3 & 0 & 0 & 0 \\
0 & 0 & 0 & 0 & 0 & 0 & 3 & 0 & 0 \\
0 & 0 & 0 & 0 & 0 & 0 & 0 & 3 & 0 \\
-1 & 0 & 0 & 0 & -1 & 0 & 0 & 0 & 2
\end{bmatrix} \otimes  M_1, \ \ 
A  =  \begin{bmatrix}
C_3 & 0 & 0 \\
0 & C_3 & 0 \\
0 & 0 & C_3
\end{bmatrix}  \ \text{and} \ \  C_3 = \begin{bmatrix}
B_{11} &  B_{12} &  B_{13} \\[0.5em]
B_{12}^\top &  B_{22} &  B_{23} \\[0.5em]
B_{13}^\top &  B_{23}^\top &  B_{33}
\end{bmatrix},
\end{eqnarray*}}
where $M_1, B_{ij}
\in \mathbb{R}^{N_h/9 \times N_h/9}, \text{ for } i,j \in \{1,2,3 \}$  are defined, again, according to the chosen numerical discretisation method. In particular, the blocks $B_{ij}$ are directly derived from the discrete forms $a^h_{ij}(\cdot,\cdot)$ for $i,j \in \{1,2,3\}$.
We can now characterise the algebraic structure of $V^\top A^* V$, i.e. the matrix that represents the reduced system in dimension~$d$.

\begin{corollary}[Reduced Matrix Structure] \label{corol_d}
Let (\ref{pb_study}) be the algebraic formulation of the pseudo-stress Stokes problem. Let $B_{ii}$ be the block of the matrix $A$ corresponding to the diagonal
component $\sigma_{ii}$ of the pseudo-stress tensor, and let $V$ be defined as in (\ref{Vmatrix}). Then, the matrix
representing the reduced problem has the expression
$V^\top A^* V = \frac{\alpha}{d}\sum_{i=1}^d B_{ii}$,
where $\sum_{i=1}^d B_{ii}$ coincides with the stiffness matrix of the discrete Laplacian in dimension $d$.
\end{corollary}

\begin{proof}
\textcolor{black}{Since $\mathrm{span}(V) = \ker(M)$, by definition~\eqref{Vmatrix}, we have that $V^\top M V = 0$, and therefore $
V^\top A^* V = V^\top (M + \alpha A) V = \alpha\, V^\top A V.$ We now compute $V^\top A V$ explicitly. Recall that $V$ 
contains an identity block in correspondence with each diagonal
component $\sigma_{ii}$, $i=1,\dots,d$, and zero blocks elsewhere,
and that $A$ has a block-diagonal structure with $d$ copies of 
$C_d$ on the block diagonal. Then, the $i$-th identity block in $V^\top$
extracts the $i$-th block-row of the $i$-th copy of $C_d$.  
This yelds
\begin{equation*}
  V^\top A = \frac{1}{\sqrt{d}} \ 
  \bigl[\,
    \underbrace{B_{11} \ \cdots \ B_{1d}}_{\text{1st block-row of }C_d} 
    \quad \cdots \quad
    \underbrace{B_{d1} \ \cdots \ B_{dd}}_{d\text{-th block-row of }C_d}
  \,\bigr],
\end{equation*}
where $B_{ij} = B_{ji}^\top$ for $i \neq j$. Multiplying on the right by $V$, all off-diagonal blocks 
$B_{ij}$ with $i \neq j$ are annihilated. Therefore, only the diagonal blocks $B_{ii}$ survive, 
and we obtain $
  V^\top A^* V = \frac{\alpha}{d}\sum_{i=1}^{d} B_{ii}.
$}
\end{proof}
For $d=2$, we have that $ 
V =   \frac{1}{\sqrt{2}}
\begin{bmatrix}
1 \ \
0 \ \
0 \ \
1
\end{bmatrix}^\top \otimes I_2,$
where $I_2 \in \mathbb{R}^{\frac{N_h}{4}\times\frac{N_h}{4}}$ denotes the identity matrix. While the matrix representing the reduced system satisfies $
V^\top A^* V
=
\frac{\alpha}{2}\,(B_{11} + B_{22}).$
Similarly, for $d=3$ we have that $V = \frac{1}{\sqrt{3}}
\begin{bmatrix}
1 \ \
0 \ \
0 \ \
0 \ \
1 \ \
0 \ \
0 \ \
0 \ \
1 
\end{bmatrix}^\top \otimes  I_3, $ where $I_3$ now denotes the identity matrix of size $N_h/9$. Moreover, we obtain that the matrix of the reduced system has the following expression: $
V^\top A^* V
=
\frac{\alpha}{3}\,(B_{11} + B_{22} + B_{33}).$

\subsection{Examples of spatial discretisation methods} \label{spatial_discr}

The fully-discrete formulation \eqref{weak_dg} has been introduced in an
abstract form, without specifying the spatial discretisation. This allows the
framework to accommodate different numerical methods for the approximation of
the spatial operators. In this section, we illustrate two possible spatial discretisations fitting
the abstract setting \eqref{weak_dg}. In particular, we consider a
continuous Galerkin finite element method and a discontinuous Galerkin method on
polytopal meshes (PolyDG) (see \cite{CangianiDongGeorgoulisHouston_2017, CangianiHoustonArticle} for further details). In each case,
we define the corresponding discrete space $\boldsymbol{V}_h$ and the
associated discrete bilinear and linear forms $m^h(\cdot,\cdot)$,
$a^h_{ij}(\cdot,\cdot)$ and $f^h_{ij}(\cdot)$.  In the following, we use the short-hand notation $(\cdot,\cdot)_{\mathcal{T}_h} = \sum_{\kappa\in \mathcal{T}_h}\int_{\kappa} \cdot$ and $ \langle \cdot,\cdot \rangle_{\mathcal{F}_h} = \sum_{F\in \mathcal{F}_h}\int_{F} \cdot$.

\subsubsection{Conforming Galerkin discretisation} 
In this section, we present the finite element framework used for a conforming Galerkin discretisation of the problem.
Let $\mathcal{T}_h$ be a conforming triangulation 
of the domain $\Omega$, which means that there exists a constant $C>0$ such that $
\sup_{\kappa \in \mathcal{T}_h} h_\kappa \le h, \ 
\sup_{\kappa \in \mathcal{T}_h} \frac{h_\kappa}{\rho_\kappa} \le C,$
where $h_\kappa$ and $\rho_\kappa$ denote, respectively, the diameter and the inradius of
the simplex $\kappa$.
For an integer $\ell \ge 0$, we denote by $\mathbb{P}_\ell(\kappa)$ the space of real
polynomials of degree at most $\ell$ on $\kappa$, and by $\mathrm{RT}_\ell(\kappa)$ the 
Raviart--Thomas space \cite{Ndlec1980MixedFE, raviartmixed} of order $\ell$ on $\kappa$.  
Analogously, $\mathrm{BDM}_{\ell}(\kappa)$ denotes the Brezzi--Douglas--Marini space \cite{BrezziFortin, BrezziMarini}
of order $\ell$ on $\kappa$. Using these local spaces, we can introduce the conforming approximation 
$\bm V_h \subset \bm V$ 
defined as
\[
\boldsymbol V_h 
= \big\{ \boldsymbol v_h \in \boldsymbol V 
     \,:\, \boldsymbol v_h|_\kappa \in \mathrm{RT}_\ell(\kappa) \ 
     \text{ or } \ \mathrm{BDM}_{\ell+1}(\kappa),
     \ \forall\, \kappa \in \mathcal{T}_h \big\}.
\]
Independently of the specific finite element space adopted, the discrete 
variational problem retains the same structure as the continuous one.  
In particular, the corresponding fully-discrete approximation has the form \eqref{weak_dg}, where 
\begin{eqnarray*}
    m^h(\sigma, \tau)  =  ( \sigma, \tau)_{\mathcal{T}_h},  \quad
    \quad a^h_{ij}(\sigma, \tau) =  (\partial_i \sigma, \partial_j \tau)_{\mathcal{T}_h}, \quad \quad
    f^h_{ij}(\tau) & = & (F^*_{ij}, \tau)_{\mathcal{T}_h}.
\end{eqnarray*}

\begin{remark}
    \textcolor{black}{The well-posedness and error analysis of conforming Galerkin 
  discretisations based on  RT and BDM finite elements 
  for the pseudo-stress formulation of Stokes-type problems have been 
  addressed in~\cite{Cai2009, Cai2007}.}
\end{remark}

\subsubsection{PolyDG discretisation}\label{sec:polydg}
We now consider the PolyDG method, a discontinuous Galerkin scheme defined on general polygonal or polyhedral meshes (see~\cite{CangianiDongGeorgoulisHouston_2017} for a comprehensive monograph). 
%
\textcolor{black}{A detailed stability and convergence analysis of the PolyDG approximation for the pseudo-stress formulation can be found in \cite{Cancrini}; see also Remark~\ref{remark1} for the assumptions and error estimates relevant to the present work.}
Let $\mathcal{T}_h$ be a polytopal mesh of the domain $\Omega$, i.e.,   $\mathcal{T}_h = \textcolor{black}{\bigcup_{\kappa} \kappa}$, being  $\kappa$ a general polygon ($d = 2$) or polyhedron ($d = 3$). 
Given a polytopal element $\kappa$, we define by $\left|  \kappa \right|$ its measure and by \textcolor{black}{$h_\kappa$} its diameter, and set \textcolor{black}{$h = \max_{\kappa \in \mathcal{T}_h} h_\kappa$}.
 We let a polynomial degree \textcolor{black}{$p_\kappa \ge 1$} be associated with each element $\kappa \in \mathcal{T}_h$ and we denote by $p_h : \mathcal{T}_h \rightarrow \mathbb{N^*} = \{n \in \mathbb{N} : n \ge 1\}$ the piecewise constant function such that $(p_h)|_{\kappa} = \textcolor{black}{p_\kappa}$. Then, we define the discrete space
$\boldsymbol{V}_h = [\textit{P}_{p_h} (\mathcal{T}_h)]^{d\times d}, \ \text{where} \ \ \textit{P}_{p_h} (\mathcal{T}_h) = \Pi_{\kappa \in \mathcal{T}_h} \mathbb{P}_{p_\kappa} (\kappa),$ 
and $\mathbb{P}_{\ell} (\kappa)$ is the space of piecewise polynomials in $\kappa$ of total degree less than or equal to $\ell \geq 1.$ 
We refer the reader to \cite{Cancrini, Cangiani2022, CangianiDongGeorgoulisHouston_2017} for the main assumptions on $\mathcal{T}_h$.
 We define an interface as the intersection of the ($d - 1$)-dimensional faces of any two neighboring
elements of $\mathcal{T}_h$. In two dimensions ($d=2$), the interfaces of any element $\kappa \in \mathcal{T}_h$ are line segments, i.e., $(d-1)$-dimensional simplexes. In three dimensions ($d=3$), this will not be the case; we assume that each interface of an element $\kappa \in \mathcal{T}_h$ can be partitioned into a collection of co-planar triangles. Accordingly, we use the term ``face'' to denote a $(d-1)$-dimensional simplex (line segment for $d=2$ or triangle for $d=3$) belonging to the boundary of $\kappa$.
We also decompose the set of faces as $\mathcal{F} = \mathcal{F}_h^I \cup \mathcal{F}_h^D  \cup \mathcal{F}_h^N$, where $\mathcal{F}_h^I$ contains the internal faces and $\mathcal{F}_h^D$ and $\mathcal{F}_h^N$ the faces of the Dirichlet and Neumann boundary conditions, respectively.
 Finally, for sufficiently piecewise smooth vector- and tensor-valued fields $\bm{v}$ and $\bm{\tau}$, respectively, and for any pair of neighbouring elements $\kappa^+$ and $\kappa^-$ sharing a face $F\in \mathcal{F}_h^I$,  we introduce the jump and average operators $ [[ \bm{v} ]]  = \bm{v}^+\otimes\bm{n}^++\bm{v}^-\otimes\bm{n}^-, [[\boldsymbol{\tau}]] = \boldsymbol{\tau}^+ \boldsymbol{n}^+ +  \boldsymbol{\tau}^- \mathbf{n}^-, \
 \{\!\{\bm{v}\}\!\} = \frac{(\bm{v}^+ + \  \bm{v}^-)}{2}, \{\!\{\boldsymbol{\tau}\}\!\} = \frac{(\boldsymbol{\tau}^+  + \  \boldsymbol{\tau}^-)}{2},$
where $\otimes$ is the tensor product in $\mathbb{R}^d$, $\cdot^{\pm}$ denotes the trace on $F$ taken within $\kappa^\pm$, and $\bm{n}^\pm$ is the outer normal vector to $\partial \kappa^\pm$. 
Accordingly, on boundary faces  $F \in \mathcal{F}_h^D  \cup \mathcal{F}_h^N$, 
we set $[[\bm{v}]] = \bm{v}\otimes\bm{n}, \,\,
[[\boldsymbol{\tau}]] = \boldsymbol{\tau} \boldsymbol{n}, \,\, {\rm and }  \,\,
\{\!\{\bm{v}\}\!\} =  \bm{v}, \,\,
\{\!\{\bm{\tau}\}\!\} =  \bm{\tau}.$
In the following, we use $\nabla_h \cdot$ to denote the element-wise divergence operator.
The corresponding fully-discrete PolyDG approximation has the form \eqref{weak_dg}, where in this case we define
\begin{eqnarray*}
    m^h(\sigma, \tau) & = & ( \sigma, \tau)_{\mathcal{T}_h},  \qquad f^h_{ij}(\tau)  =  (F^*_{ij}, \tau)_{\mathcal{T}_h}, \\ 
    \quad a^h_{ij}(\sigma, \tau) & = & (\partial_i \sigma, \partial_j \tau)_{\mathcal{T}_h} - \langle \{\!\{\partial_i \sigma\}\!\} , [[\tau]] \rangle_
    {\mathcal{F}_h^{I,N}} \\& & - \langle \{\!\{\partial_j \tau\}\!\} , [[\sigma]] \rangle_
    {\mathcal{F}_h^{I,N}}    
    + \langle \gamma_e[[\sigma]] , [[\tau]] \rangle_{\mathcal{F}_h^{I,N}}.
\end{eqnarray*}
Here, $\mathcal{F}_h^{I,N} = \mathcal{F}_h^{I} \cup \mathcal{F}_h^{N}$ and the stabilization function $\gamma_e: \mathcal{F}_h^{I,N}\rightarrow\mathbb{R}_+$ is defined as a function of the penalty coefficient $\alpha^*>0$ as follows:
\begin{equation*}\label{def:penalty}
    \gamma_e(\bm x) = \begin{cases}
        \alpha^* \max_{\kappa \in \{ \kappa^+,\kappa^-\}} \frac{p_\kappa^2}{h_\kappa}, & \bm x \in e, e \in \mathcal{F}^I_h, e \subset \partial \kappa^+ \cap \partial \kappa^-,  \\
        \alpha^* \frac{p_\kappa^2}{h_\kappa}, & \bm x \in e, e \in \mathcal{F}^N_h, e \subset \partial \kappa^+ \cap \partial \Gamma_N.
    \end{cases}
\end{equation*}

{\begin{remark}\label{remark1}
    Let us recall the main assumptions and results from the error analysis of the fully-discrete PolyDG scheme \cite{Cancrini}. Let the exact solution $\bm \sigma$ be sufficiently smooth and $\bm \sigma_h^{n}$ be the solution of \eqref{pb_study} with $\alpha = \Delta t$ (implicit Euler) and for a sufficiently large penalty coefficient $\alpha^*$. If $(h, p)$ are quasi uniform, the error in the energy norm satisfies $\| \bm \sigma -\bm \sigma_h \|_{E} \sim \Delta t + h^p$. 
\end{remark}

\section{Numerical solvers} \label{sec:solvers}

\textcolor{black}{In this section, we present a numerical framework for the solution of the linear system~\eqref{pb_study}.
Since $A^*$ is symmetric and 
positive definite, the CG method is the natural choice. However, as 
discussed in Section~\ref{reduced_continuous}, the conditioning of $A^*$ 
deteriorates as $\Delta t \to 0$, motivating the use of a deflated 
CG strategy that projects out the directions in $\ker(M)$ and 
stabilises the iteration count~\cite{CancriniCiaramella}. The structure of the resulting inner system naturally 
motivates the use of a multigrid solver, where a RAS smoother exploits the 
block structure of the discretisation. Since the inner system is solved 
only approximately, the deflation projection applied at each outer 
iteration is itself inexact, and the corresponding correction varies 
from one iteration to the next depending on the inner solver 
tolerance. As a result, if the tolerance of the inner solver is not appropriately related to the outer residual, the 
classical CG method is no longer guaranteed to be robust. This 
motivates the use of a flexible conjugate gradient algorithm, 
specifically designed to accommodate variable system matrices.}
In summary, our solution strategy combines deflation to achieve robustness with respect to $\Delta t$, a multigrid method to efficiently solve the inner system, and a flexible CG scheme to ensure stable convergence of the overall algorithm. In Sections~\ref{review_deflated} - \ref{sec_fcg}, we introduce the main building blocks of the proposed solver: the deflated conjugate gradient method, the inner linear system and its multigrid solution strategy, and the flexible CG framework. The corresponding algorithms are presented in Section~\ref{num_algo}.

\subsection{The deflated conjugate gradient method} \label{review_deflated}

Consider a linear system of equations of the form (\ref{pb_study}), 
where $A^* \in \mathbb{R}^{N_h \times N_h}$ is symmetric and positive definite and $\bm \sigma_h^{n}, \ \bm f^{*} \in \mathbb{R}^{N_h}$. 
To simplify the notation, we denote by $\bm x$ the solution $\bm \sigma_h^{n}$ of the linear system.  
\textcolor{black}{Since the speed of convergence of CG depends on the distribution of the eigenvalues of $A^*$, the idea of deflation \cite{Kahl, Nicolaides} is to ``hide'' parts of the spectrum of $A^*$ from CG such that the CG iteration operates on an equivalent system with a significantly reduced condition number compared to 
the original one.} The spectral components that are hidden depend on the chosen deflation subspace $\mathbf{S} \subset  \mathbb{R}^{N_h}$, and the resulting improvement in convergence speed is entirely governed by this choice. 
Given $\mathbf{S}^{\perp_{A^*}}=(A^*\mathbf{S})^\perp$ the $A^*-$orthogonal complement of $\mathbf{S}$, it is possible to split the solution $\bm x$ into a component in $\mathbf{S}$ and a component in $\mathbf{S}^{\perp_{A^*}}$ via the $A^*-$orthogonal projection $\pi_{A^*}(\mathbf{S}) \in \mathbb{R}^{N_h \times N_h}$ onto $\mathbf{S}$. Given a matrix $V\in \mathbb{R}^{N_h \times m}$ whose columns form a basis for $\mathbf{S}$, and $\pi_{A^*}(\mathbf{S}) = V(V^\top A^*V)^{-1}V^\top A^*$, we obtain 
\begin{align}\label{x_defl}
   \bm x = (I-\pi_{A^*}(\mathbf{S})) \ \hat{\bm x} + V(V^\top A^*V)^{-1}V^\top  \bm f^{*},
\end{align}
where $\hat{\bm x}$ is the solution of the so-called \textit{deflated linear system}, i.e.
\begin{equation} \label{defl_sys}
    A^*(I-\pi_{A^*}(\mathbf{S})) \ \hat{\bm x} = (I-\pi_{A^*}(\mathbf{S}))^\top \bm f^{*}.
\end{equation}
Thus, to obtain $\bm x$, one needs to compute the solution $\hat{\bm x}$ of the deflated system. \textcolor{black}{We note that $A^*(I-\pi_{A^*}(\mathbf{S})) = (I-\pi_{A^*}(\mathbf{S}))^\top A^*(I-\pi_{A^*}(\mathbf{S}))$, as shown in \cite{Kahl}. Then the matrix $A^*(I-\pi_{A^*}(\mathbf{S}))$ is symmetric positive semi-definite and so we can apply the CG method.} \textcolor{black}{Matrix singularity poses no obstacle to the standard CG iteration, provided that (\ref{defl_sys}) is consistent, i.e., the right hand side $(I-\pi_{A^*}(\mathbf{S}))^\top\bm f^{*}$ is in the range of $A^*(I-\pi_{A^*}(\mathbf{S}))$,  which has been 
verified in \cite[Lemma~2.1]{Kahl}.}
To summarise, one interprets deflated CG as the standard CG algorithm applied to the deflated system (\ref{defl_sys}).
\textcolor{black}{Let $\lambda_1 \ge \cdots \ge \lambda_n \ge 0$ denote the eigenvalues of the self-adjoint and positive semi-definite matrix $A^*(I-\pi_{A^*}(\mathbf{S}))$, and $k \in \mathbb{N}$ the largest index such that $\lambda_k \neq 0$. Then the errors of the CG iterates satisfy
\[
\|e_i\|_{A^*} \le 2 \left( \frac{\sqrt{\kappa_{\mathrm{eff}}}-1}{\sqrt{\kappa_{\mathrm{eff}}}+1} \right)^i \|e_0\|_{A^*},
\quad i = 0,1,2,\dots,
\]
where $\kappa_{\mathrm{eff}} = \lambda_1 / \lambda_k$ is the effective condition number of the deflated matrix $A^*(I-\pi_{A^*}(\mathbf{S}))$, in contrast to the condition number $\kappa$ of $A^*$ \cite{Kahl}.}

Note that applying the deflated CG method requires computing the inverse of $V^\top A^* V$ in (\ref{x_defl}) in the definition of $\pi_{A^*}(\mathbf{S})$. 
In particular, in the deflated CG method, each iteration requires solving the \textit{inner linear system}
\begin{align}\label{inner}
  (V^\top A^*V) \ \bm z_{k+1} = V^\top A^* \bm r_{k+1},
\end{align}
where $\bm r_{k+1}$ is the current residual, and $\bm z_{k+1}$ is the solution of the 
inner system. When the dimension of the subspace $\mathbf{S}$ is large, solving the inner problem~(\ref{inner}) exactly may become too costly. In such cases, it is reasonable to approximate its solution by means of an iterative solver.
Previous analyses~\cite{Kahl, Nabben2006} have shown that deflation methods may react strongly to the level of accuracy achieved in the inner computation. The outer iteration is terminated once
\[
\|\boldsymbol r_i\|_2 \le \tau\, \|\boldsymbol f^*\|_2 = \varepsilon,
\qquad 0 < \tau \ll 1,
\]
and we seek an analogous stopping rule for the inner iteration, expressed as
\[
\|\boldsymbol r_i^{c}\|_2 \le \tau^c\, \|\boldsymbol f^c\|_2.
\]
The results reported in~\cite{Nabben2006} indicate that a fixed tolerance of the form
\begin{equation}\label{fixed_tol}
\tau^c = \varepsilon\, c_F,
\qquad 0 < c_F \le 1,
\end{equation}
is often sufficient. Furthermore, as pointed out in~\cite{Kahl}, a high level of accuracy in the inner problem is crucial mainly during the first few outer iterations, whereas it can be gradually relaxed thereafter. This observation motivates the choice of a variable threshold such as
\begin{equation}\label{tol}
    \tau^c = c_A\, \frac{\varepsilon}{\|\boldsymbol r_i\|},
\end{equation}
for some constant $c_A > 0$. We refer to \eqref{tol} as an \emph{adaptive tolerance rule}~\cite{Kahl}: the relative accuracy required for the inner system becomes less stringent as the outer residual $\boldsymbol r_i$ decreases, while still ensuring that the outer CG iteration converges to the target precision.

\subsection{Multigrid solver for the inner system} \label{analysis_inner}
We then choose $V$ as the kernel of the matrix $M$: this choice is motivated by the fact that the kernel of $M$ captures the problematic spectral components of the system, and deflating this subspace effectively removes the directions responsible for slow convergence of CG as $\Delta t \to 0$. This corresponds precisely to defining the matrix V as in (\ref{Vmatrix}) and the deflation subspace $\mathbf{S}^h = \ker(\mathcal{M}^h)$. Then, the inner system corresponds to the discrete reduced system defined in Definition~\ref{reduced_discr}. From the analysis carried out in Section~\ref{reduced_discrete}, it follows that the inner system corresponds to a Laplacian problem. This feature immediately suggests the use of multigrid methods as efficient solvers for the associated inner problem. 
In our numerical tests, the inner system is handled by a multigrid iterative scheme employing a W-cycle strategy, and as a smoother we use a RAS iteration. We now introduce an iterative multigrid strategy to efficiently solve the inner problem (\ref{inner}). 
Let us rewrite the inner system (\ref{inner}) as 
\begin{align}\label{inner1}
    A^{c}_J \ \bm z_J = \bm f^c_J,
\end{align}
and consider a sequence of finite-dimensional inner product spaces $\boldsymbol{W}_1,   \boldsymbol{W}_2, \dots , \boldsymbol{W}_J$ (possibly non-nested). The inner product on $\boldsymbol{W}_k$ is denoted by $(\cdot, \cdot)_k$. Then, we define the symmetric positive definite bilinear form $\mathcal{A}^c_k(\cdot, \cdot)$ on $\boldsymbol{W}_k\times \boldsymbol{W}_k$ for $k=1, \dots,J$. Our goal is to efficiently solve (\ref{inner1}) at the finest level (i.e. for $k = J$); indeed, the multigrid algorithm that will be presented generates iterative procedures for solving the problem on $\boldsymbol{W}_J$. 
Let $A^c_k: \boldsymbol{W}_k \rightarrow \boldsymbol{W}_k$ be the (symmetric positive definite) operator such that 
\begin{align*}
    (A^c_k \boldsymbol{u}, \boldsymbol{v}) = \mathcal{A}^c_k(\boldsymbol{u}, \boldsymbol{v}) \qquad \text{for all } \boldsymbol{u}, \boldsymbol{v} \in \boldsymbol{W}_k.
\end{align*}
Then the corresponding algebraic formulation at the finest level reads as $
    A^{c}_J \ \bm z_{J} = \bm f^c_J,$ i.e. (\ref{inner1}).
We now recall the main ingredients of the multigrid algorithm \cite{Bramble1992, Bramble1991TheAO}: a \textit{smoother} operator, and \textit{prolongation} and \textit{restriction} operators. 
Denoting by $S_k$ the \textit{smoother}, 
the $j$-th smoothing iteration reads as
\begin{align*}
   \boldsymbol{z}^{(j+1)}_k = \boldsymbol{z}^{(j)}_k + S_k (\boldsymbol{f}^c_k - A^{c}_k \ \boldsymbol{z}^{(j)}_k), \quad \quad j = 1,2, \dots
\end{align*}
Now, since the spaces are (possibly) non-nested, i.e., $\boldsymbol{W}_{k-1} \not \subset \boldsymbol{W}_{k}$, 
the most natural way to define the prolongation operator is through the 
$L^2$-projection; see, e.g. \cite{Antonietti2017VcycleMA}. Thus, the operator 
$I^{k}_{k-1} : \bm W_{k-1} \to \bm W_k$ is then defined by
\begin{equation*}
  \big( I^{k}_{k-1} \bm u_H , \bm v_h \big)_{L^2(\Omega)} 
  = \big( \bm u_H , \bm v_h \big)_{L^2(\Omega)}
  \qquad \forall\, \bm v_h \in \bm W_k .
\end{equation*}
The restriction operator $I^{\,k-1}_{k} : \bm W_k \to \bm W_{k-1}$ is then defined 
as the adjoint of $I^{k}_{k-1}$ with respect to the $L^2(\Omega)$-inner product, 
that is,
\begin{equation*}
  \big( I^{\,k-1}_{k} \bm v_h , \bm u_H \big)_{L^2(\Omega)} 
  = \big( \bm v_h , I^{k}_{k-1} \bm u_H \big)_{L^2(\Omega)}
  \qquad \forall\, \bm u_H \in \bm W_{k-1} .
\end{equation*}
To enhance the performance of the W-cycle, we employ a RAS–type smoothing method~\cite{Ciaramella}. Consider the linear system (\ref{inner1}), and let the set of unknowns be partitioned into blocks $\mathcal{B}_i$, with $i=1, \dots, n_b,$ 
so that every component $z_j$ of the global vector $\bm z$ belongs to at least one 
block. For each block $\mathcal{B}_i$, we denote by 
$R_i : \mathbb{R}^n \rightarrow \mathbb{R}^{n_i}$ 
the operator that extracts from the global vector the subvector associated with 
$\mathcal{B}_i$, where $n_i$ is the size of the block. With a slight abuse of notation, 
we use the same symbol for both the operator and its matrix representation. 
Its transpose, $R_i^{\top}$, acts as a zero-extension prolongation operator from 
$\mathbb{R}^{n_i}$ to $\mathbb{R}^n$. 
The corresponding local matrix is defined as 
$(A^c_k)_i = R_i \ A^c_k \ R_i^{\top}$,
which represents the restriction of $A^c_k$ to the block $\mathcal{B}_i$. 
Thus, each block requires solving the small local problem
$(A^c_k)_i \ \bm z_i^{(j)} = R_i \ (\bm f^c_k - A^c_k \ \bm z_k^{(j)})$.
Once the local problems have been solved, the RAS iteration updates the global solution as follows:
\[
\bm z^{(j+1)}_k
= \bm z^{(j)}_k
+ \sum_{i=1}^{n_b} 
   R_i^{\top} D_i \ \bm z_i^{(j)},
\]
where the matrix $D_i$ denotes a diagonal (partition of unity) weighting operator. In general, its diagonal entries $d_{jj}$, $j=1,\ldots,n$, are chosen as the
reciprocal of the number of blocks that include the unknown $z_j$. \textcolor{black}{Then, the RAS smoother operator is defined as $
S_k = \sum_{i=1}^{n_b} R_i^{\top} D_i (A_k^c)_i^{-1} R_i$. We define the blocks $\mathcal{B}_i$ as the collection of degrees of freedom associated with the $i$-th mesh element and all degrees of freedom belonging to its immediate neighboring elements, resulting in a one-layer overlap.}
This choice balances computational efficiency 
with the smoothing effectiveness of the RAS iteration within the W-cycle.

\subsection{Flexible conjugate gradient method} \label{sec_fcg}
In the proposed approach, the system operator is not applied exactly at each iteration. Indeed, due to the deflation strategy, each outer iteration requires the solution of an inner linear system, which is handled approximately by a multigrid method. As a consequence, the effective operator seen by the outer iteration varies from step to step.
The classical convergence theory of the CG method, however, relies on the assumption that the underlying operator is fixed, symmetric positive definite, and applied exactly at every iteration. When this assumption is violated, for instance due to inexact inner solves or variable tolerances, standard CG may lose robustness and exhibit degraded convergence.
To address this issue, flexible Krylov methods are required. In particular, we employ the Flexible Conjugate Gradient (FCG) method~\cite{Notay}, which is specifically designed to accommodate variable operators. The main distinction with respect to classical CG lies in the explicit $A$-orthogonalization of each new search direction against the previous ones. At iteration $i$, the new direction is orthogonalized with respect to a selected number of earlier directions, determined by a sequence of truncation parameters $\{m_i\}$. In the absence of truncation (FCG($\infty$)), orthogonality is enforced with all previously generated directions, i.e., $m_i = i$ for all $i$.
In Algorithm~\ref{algoFCG} we report the untruncated FCG scheme (FCG($\infty$)) for solving the system under consideration, i.e. $A^* \bm x=\bm f^{*}$, where $A^* \in \mathbb{R}^{N_h \times N_h}$ is symmetric positive definite, $\bm f^{*} \in \mathbb{R}^{N_h}$, and $\bm x_0 \in \mathbb{R}^{N_h}$ is an initial guess (see also Algorithm 2.1 in \cite{Notay}).







\begin{algorithm}
\caption{ \texttt{FCG} (untruncated Flexible Conjugate Gradient method)}\label{algoFCG}
\begin{small}
\begin{algorithmic}[1]
\STATE Input: rhs $\bm f^{*}$, initial guess $\bm x_0$, tolerance tol, map $(\bm d, \text{tol}) \to A^*(\bm d, \text{tol})$
\STATE Output: solution $\bm x$
\STATE $\bm r_0 = \bm f^{*} - A^*( \bm x_0, \text{tol})$;

\FOR{$i = 0,1, \dots$}
    
     
  

    
    
    \STATE $\bm d_i = \bm r_i - \sum_{k=0}^{i-1} \frac{(\bm r_i, \ A^*( \bm d_k, \text{tol}))}{(\bm d_k, \ A^*( \bm d_k, \text{tol}))} \ \bm d_k$;
    \vspace{0.2em}
    \STATE $\alpha_i = \frac{(\bm d_i, \ \bm r_i)}{(\bm d_i, \ A^*( \bm d_i, \text{tol}))}$;
    \vspace{0.2em}
    \STATE $\bm x_{i+1} = \bm x_i + \alpha_i \ \bm d_i$;
    
    \STATE $\bm r_{i+1} = \bm r_i - \alpha_i  \ A^*( \bm d_i, \text{tol})$;
   
\ENDFOR
\end{algorithmic}
\end{small}
\end{algorithm} 

\textcolor{black}{Compared to the CG approach, the untruncated FCG method incurs higher computational and memory costs, as it stores all previous search directions and performs full orthogonalization at each iteration. For a moderate number of iterations, however, this overhead remains acceptable and it is compensated by the improved robustness of the method. Moreover, the added advantage of full flexibility may justify this extra cost. In addition, by using a fully flexible, unrestarted scheme, we avoid delicate parameter tuning and obtain reliable convergence even when the operator is applied inexactly (see also Remark~\ref{remark_fcg}).}
\begin{remark}
\textcolor{black}{An alternative approach to handle the loss of robustness of 
standard CG is the use of the Generalized Minimum Residual (GMRES) 
method~\cite{Saad2003} as outer solver, which builds an orthonormal basis 
of the Krylov subspace via the Arnoldi process and minimises 
the residual over the entire subspace at each iteration. As 
discussed in Remark~\ref{rmk:gmres} below, GMRES eliminates 
the convergence issues observed with standard CG under inexact 
inner solves, confirming its robustness in this setting. The 
choice of FCG over GMRES in our framework is motivated by the 
fact that FCG retains the structure of a conjugate gradient 
method and exploits the symmetry of $A^*$.}
\end{remark}

\subsection{Numerical algorithms} \label{num_algo}
In this section, we detail the numerical algorithms that define the flexible deflated CG method and the multigrid W-cycle scheme adopted in our numerical simulations. 
We start by recalling the interpretation introduced in Section~\ref{review_deflated}: the (flexible) deflated CG method can be viewed as the application of the (flexible) CG algorithm to the deflated linear system~\eqref{defl_sys}. In Algorithm~\ref{algoFDCG} we report the flexible deflated CG scheme for solving (\ref{pb_study}), i.e., the problem $A^* \bm x=\bm f^{*}$.

\begin{algorithm}
\caption{ \texttt{FDCG} (Flexible Deflated Conjugate Gradient method)}\label{algoFDCG}
\begin{small}
\begin{algorithmic}[1]
\STATE Input: matrix $A^*$, rhs $\bm f^{*}$, initial guess $\bm x_0$, smoothing steps $m_1$ and $m_2$, tolerance tol, maximum number of iterations $n_{\max}$, matrix $V$, prolongation $I^{\,k}_{k-1}$ and restriction $I^{\,k-1}_{k}$
\STATE Output: solution $\bm x$

\STATE $Z_J = V^\top A^* V$;

\FOR{$k = J, \dots, 2$}
\STATE $Z_{k-1} = I_k^{k-1} Z_{k} I^k_{k-1}$;
\ENDFOR    

\STATE $ (\bm z, \text{tol}) \to \hat{\bm z}(\bm z, \text{tol})  :=\texttt{Multigrid}(\{Z_k\}_{k=1}^J, V^\top \bm z, 0, m_1, m_2, \text{tol}, n_{\max})$;

\STATE $(\bm z, \text{tol}) \to \pi(\bm z, \text{tol}) := V \ \hat{\bm z}(\bm z, \text{tol}) $;

\STATE  $\underline{ \texttt{Operator}\ M_D \ \texttt{and rhs} \ f_D \ \texttt{of the deflated system:}}$
\vspace{0.2em}
\STATE $(\bm z, \text{tol}) \to M_D(\bm z, \text{tol}) := A^* (\bm z   - \pi(A^* \bm z, \text{tol}))$;

\STATE $\bm f_D = \bm f^{*} - \pi^\top (\bm f^{*}, \text{tol})$;

\STATE $\hat{\bm x} = \texttt{FCG}(\bm f_D, \bm x_0, \text{tol}, M_D)$;

\STATE $\bm x = (\hat{\bm x} - \pi(A^* \hat{\bm x}, \text{tol})) + \pi(\bm f^{*}, \text{tol})$;

\end{algorithmic}
\end{small}
\end{algorithm}

Then, we turn our attention to the multigrid component of the proposed solver, which is employed to solve the inner system (\ref{inner1}). Given an initial guess $\bm z_0 \in \bm W_J$ and choosing the parameters 
$m_1, m_2 \in \mathbb{N}$ (pre-smoothing and post-smoothing steps), the multigrid W-cycle procedure used to 
approximate $\bm z_J$ is summarized in Algorithm~\ref{algo1}. 
We denote by $\bm z_{n+1} = \texttt{Wcycle}(\{A_k^c\}_{k=1}^J, \bm f^c_J, \bm z_n, J, m_1, m_2)$ the approximation obtained after one iteration of the non-nested 
W-cycle scheme. 
The recursive procedure is detailed in Algorithm~\ref{algo2}, where we also note 
that on the coarsest level $k = 1$ the problem is solved exactly. Here, we consider the general problem: find $\bm z \in \bm W_k$ such that $A^c_k \ \bm z = \bm f^c$.

\begin{algorithm}
\caption{\texttt{Multigrid} (Multigrid W-cycle iteration for the solution of (\ref{inner1}))} \label{algo1}
\begin{small}
\begin{algorithmic}[1]
\STATE Input: coarse-level matrices $\{A_k^c\}_{k=1}^J$, rhs $\bm f^c_J$, initial guess $\bm z_0$, smoothing steps $m_1$ and $m_2$, tolerance tol, maximum number of iterations $n_{\max}$
\STATE Output: solution $\bm z$, number of iterations $n$
\WHILE{$n < n_{\max}$ \textbf{and} $\|\bm f^c_J - A_J^c \bm z_n\| / \|\bm f^c_J\| > \text{tol}$}
\vspace{0.2em}
\STATE $\bm z_{n+1} = \texttt{Wcycle}(\{A_k^c\}_{k=1}^J, \bm f^c_J, \bm z_n, J, m_1, m_2)$; 
\STATE $n = n+1;$
\ENDWHILE
    
    



    
\end{algorithmic}
\end{small}
\end{algorithm}

\begin{algorithm}
\caption{\texttt{Wcycle} (one W-cycle iteration for $k \geq 2$ levels)} \label{algo2}
\begin{small}
\begin{algorithmic}[1]
\STATE Input: coarse-level matrices $\{A_k^c\}_{k=1}^J$, rhs $\bm f^c$, initial guess $\bm z_0$, level $k$, smoothing steps $m_1$ and $m_2$
\STATE Output: solution $\bm z$
\IF{$k=1$}
\STATE $\texttt{Wcycle}(\{A_k^c\}_{k=1}^J, \bm f^c, \bm z_0, 1, m_1, m_2) = (A^c_1)^{-1} \bm f^c ;$

\ELSE

\STATE  $\underline{\texttt{Pre-smoothing:}}$
\FOR{$i=1,\dots, m_1$}

\STATE $\bm z^{(i)} = \bm z^{(i-1)} + S_k (\bm f^c - A^c_k \ \bm z^{(i-1)});$

\ENDFOR 

\STATE  $\underline{\texttt{Coarse grid correction:}}$
\STATE $\bm r_{k-1} = I_k^{k-1} (\bm f^c - A^c_k \ \bm z^{(m_1)});$ 
\vspace{0.2em}
\STATE $\boldsymbol{\bar e}_{k-1} = \texttt{Wcycle}(\{A_k^c\}_{k=1}^J, \bm r_{k-1}, 0, k-1, m_1, m_2); $ \\
\vspace{0.2em}
\STATE $\boldsymbol{e}_{k-1} = \texttt{Wcycle}(\{A_k^c\}_{k=1}^J, \bm r_{k-1}, \boldsymbol{\bar e}_{k-1}, k-1, m_1, m_2); $ 
\vspace{0.2em}
\STATE $\bm z^{(m_1 + 1)} = \bm z^{(m_1)} + I^k_{k-1} \boldsymbol{e}_{k-1};$

\STATE  $\underline{\texttt{Post-smoothing:}}$
\FOR{$i=m_1+2,\dots, m_1 + m_2 +1$}

\STATE $\bm z^{(i)} = \bm z^{(i-1)} + S_k (\bm f^c - A^c_k \ \bm z^{(i-1)});$

\ENDFOR 

\STATE $\texttt{Wcycle}(\{A_k^c\}_{k=1}^J, \bm f^c, \bm z_0,k, m_1, m_2) = \bm z^{(m_1+m_2+1)} ;$

\ENDIF
\end{algorithmic}
\end{small}
\end{algorithm}

\section{Numerical results} \label{sec:num_results}
In this section, we present several numerical experiments aimed at assessing the performance of the proposed numerical framework for the solution to~\eqref{pb_study} at each time step. Numerical results are presented in both two- and three-dimensional settings, employing the spatial discretisation based on the PolyDG method introduced in Section~\ref{sec:polydg}, in order to thoroughly assess the performance of the proposed approach.

\subsection{PolyDG Finite Element discretisation - 2D case} \label{2dPolyDG}
In this section, all computations have been implemented in the open-source MATLAB library \texttt{lymph}~\cite{lymph2024}. We consider a two–dimensional problem defined on the unit square $\Omega = (0,1)^2$, and the following exact solution 
$
\boldsymbol{\sigma}(\bm x,t) = \sin(2t) \begin{bmatrix}\sin(\pi x)\sin(\pi y), \ 0,  \ 0,  \ -\sin(\pi x)\sin(\pi y)\end{bmatrix}^\top;$ then $\boldsymbol{F}(\bm x,t)$ and the boundary data are computed accordingly, in particular Dirichlet boundary conditions are imposed on the top and right boundaries, while Neumann conditions are applied on the remaining part of the boundary. For the numerical simulations we set $\mu = 1$, the penalty coefficient $\alpha^* = 10$, the polynomial degree $p = 3$, and $\alpha = \Delta t$ (implicit Euler scheme). We employ the sets of polygonal meshes illustrated in Figure~\ref{fig:Hierarchy}. In particular, the finest grids (Level~4) shown in Figure~\ref{fig:Hierarchy} contain 512 (Set~1), 1024 (Set~2), 2048 (Set~3), and 4096 (Set~4) polygonal elements. Starting from each of these initial meshes, we construct a sequence of 
non-nested partitions. Each coarse mesh is generated independently of the finer one, with the only requirement that the number of elements is approximately one quarter of that of the corresponding finer mesh, as commonly done in the literature (see, e.g., \cite{AntoniettiSarti2014, Antonietti2017VcycleMA}). As a preliminary step, we intend to frame the problem by reporting the condition number of the matrices involved, thereby providing an initial quantitative reference. In particular, in Table~\ref{tab:h_dt}, we report the condition number $\kappa(A^*)$ and the effective condition number $\kappa_{\mathrm{eff}}(A^*(I-\pi_{A^*}(\mathbf{S})))$, for the four finest grids ($h_{1}~\approx~0.082, \ h_{2} ~\approx~0.059, \ h_{3}~\approx~0.043, \ h_{4}~\approx~0.030$), and for different time steps $\Delta t$. We observe that,  for small $\Delta t$, $\kappa$ is proportional to~$1/\Delta t$ as expected, while $\kappa_{\mathrm{eff}}$ remains uniformly bounded with respect to $\Delta t$. The grey-shaded entries in the tables indicate comparable time steps and mesh sizes, and thus, according to Remark~\ref{remark1}, the regime where time and space discretisation errors are balanced.

\begin{figure}[t]
\centering
\renewcommand{\arraystretch}{0.05} 
\resizebox{0.85\textwidth}{!}{
\begin{tabular}{ccccc} 
    & \textbf{Set 1} & \textbf{Set 2} & \textbf{Set 3} & \textbf{Set 4} \\ 

    \raisebox{4.5\height}{\textbf{Level 4}} &
    \includegraphics[width=0.22\textwidth]{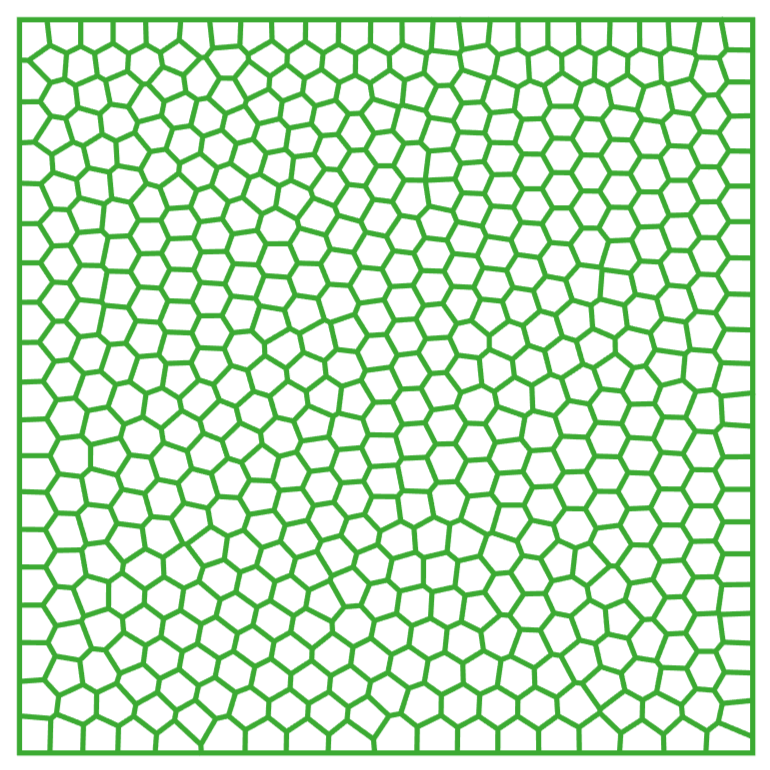} &
    \includegraphics[width=0.22\textwidth]{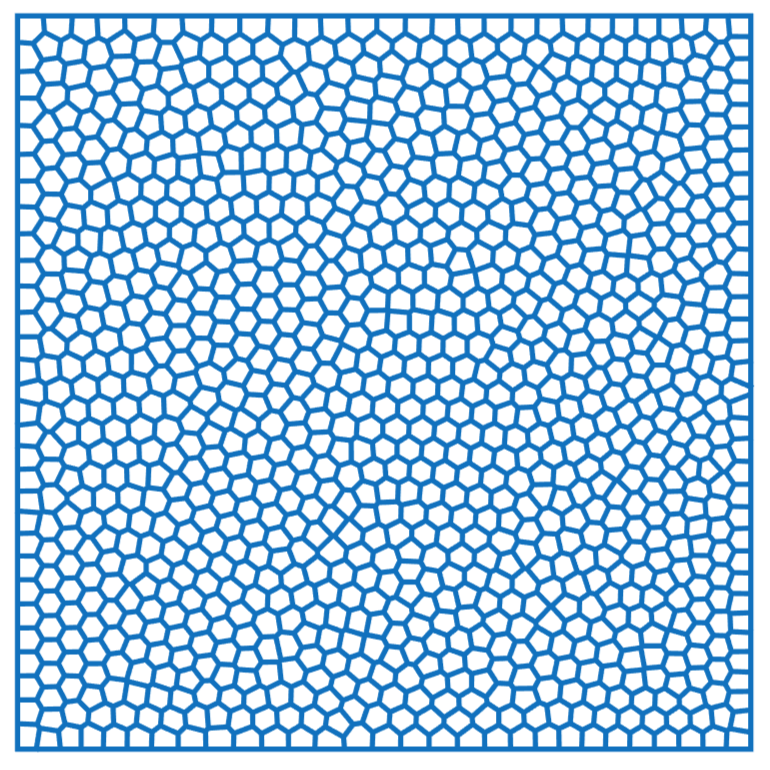} &
    \includegraphics[width=0.22\textwidth]{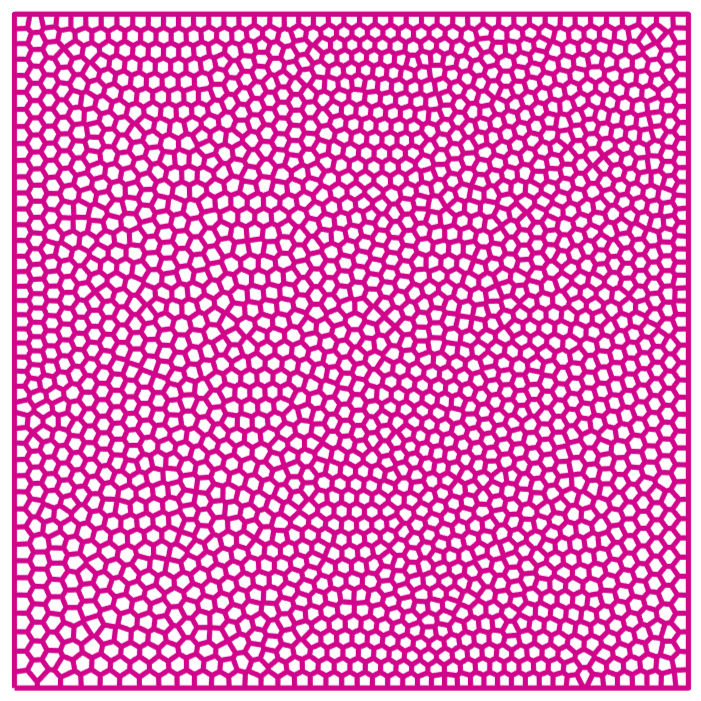} &
    \includegraphics[width=0.22\textwidth]{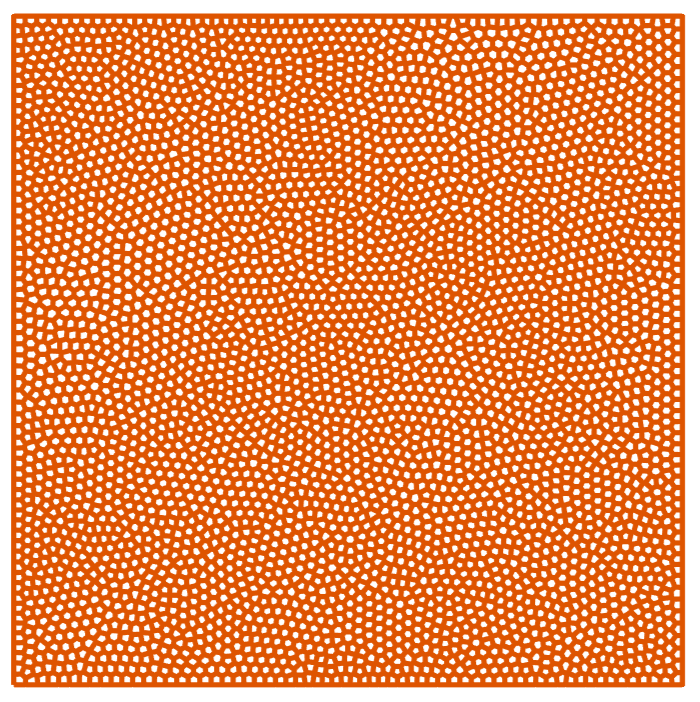} \\

    \raisebox{4.5\height}{\textbf{Level 3}} &
    \includegraphics[width=0.22\textwidth]{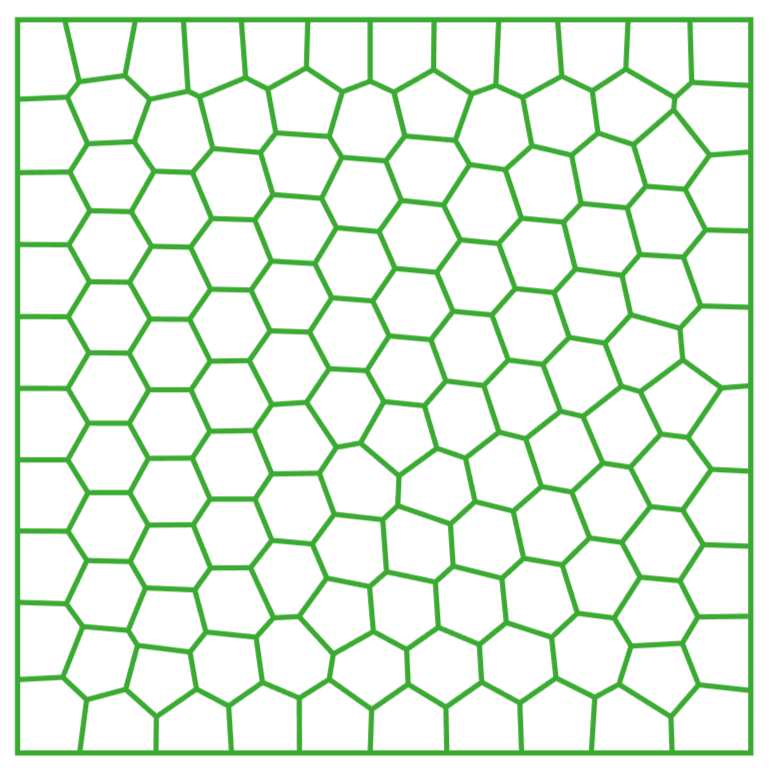} &
    \includegraphics[width=0.22\textwidth]{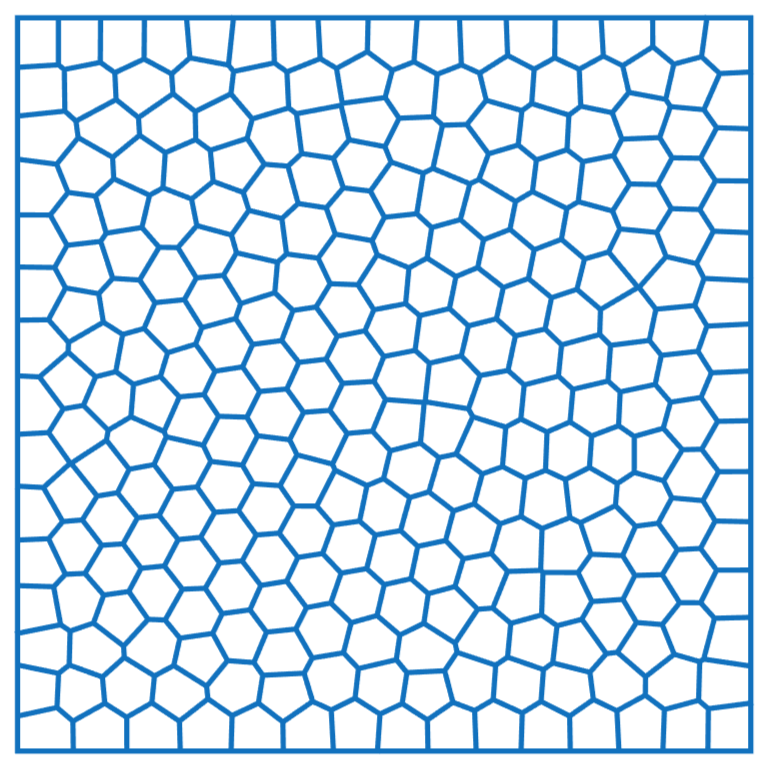} &
    \includegraphics[width=0.22\textwidth]{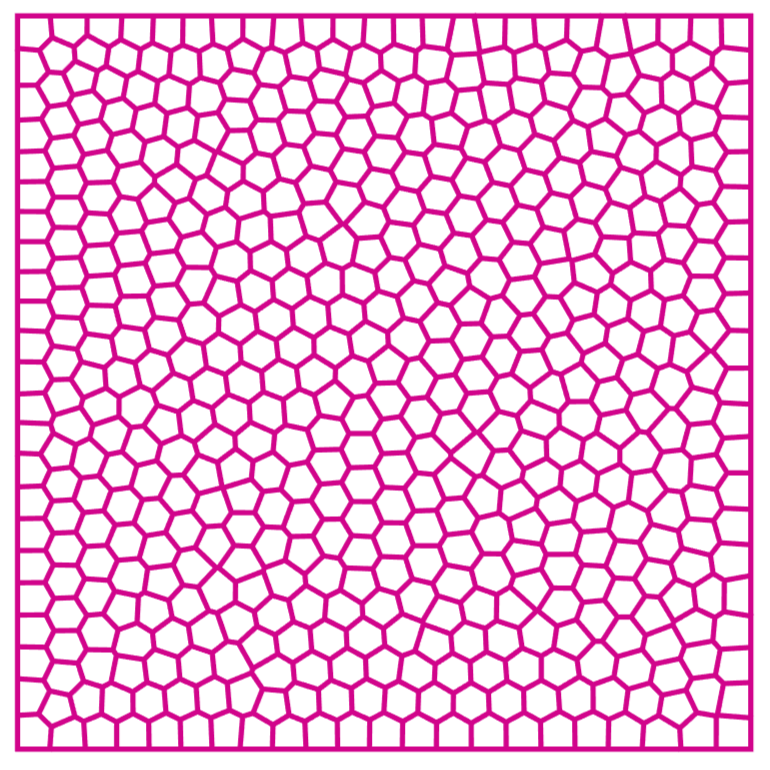} &
    \includegraphics[width=0.22\textwidth]{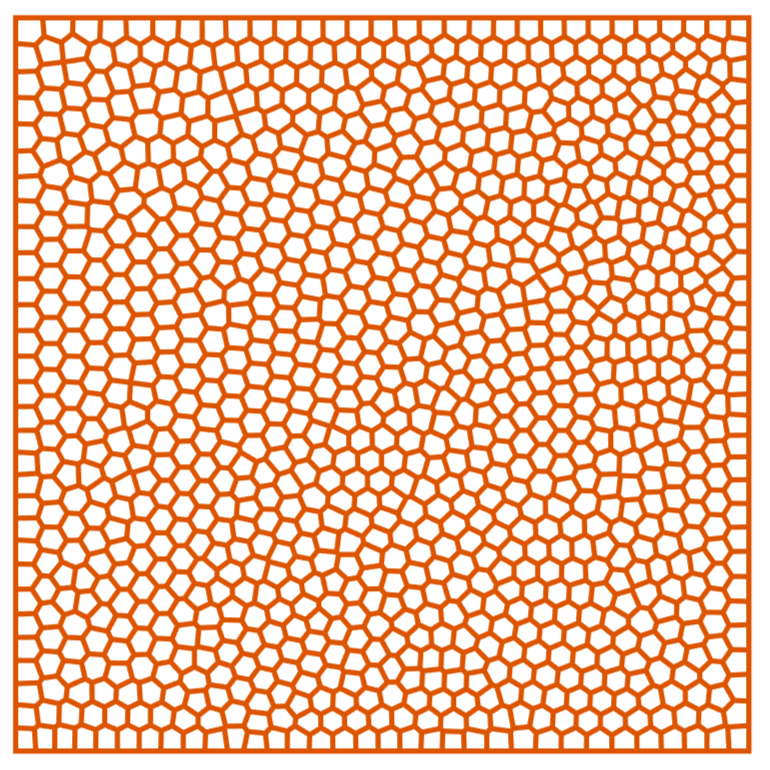} \\

    \raisebox{4.5\height}{\textbf{Level 2}} &
    \includegraphics[width=0.22\textwidth]{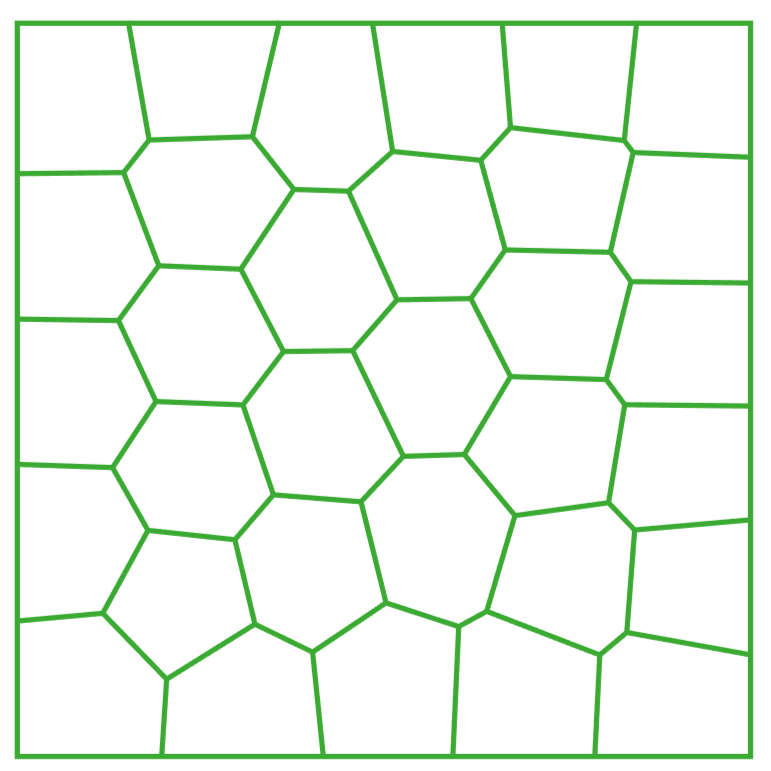} &
    \includegraphics[width=0.22\textwidth]{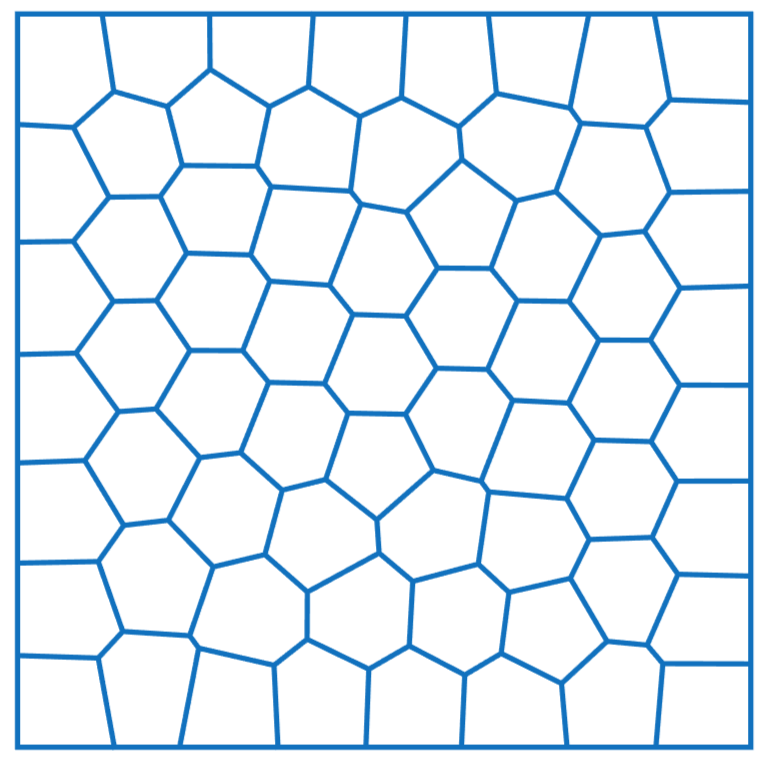} &
    \includegraphics[width=0.22\textwidth]{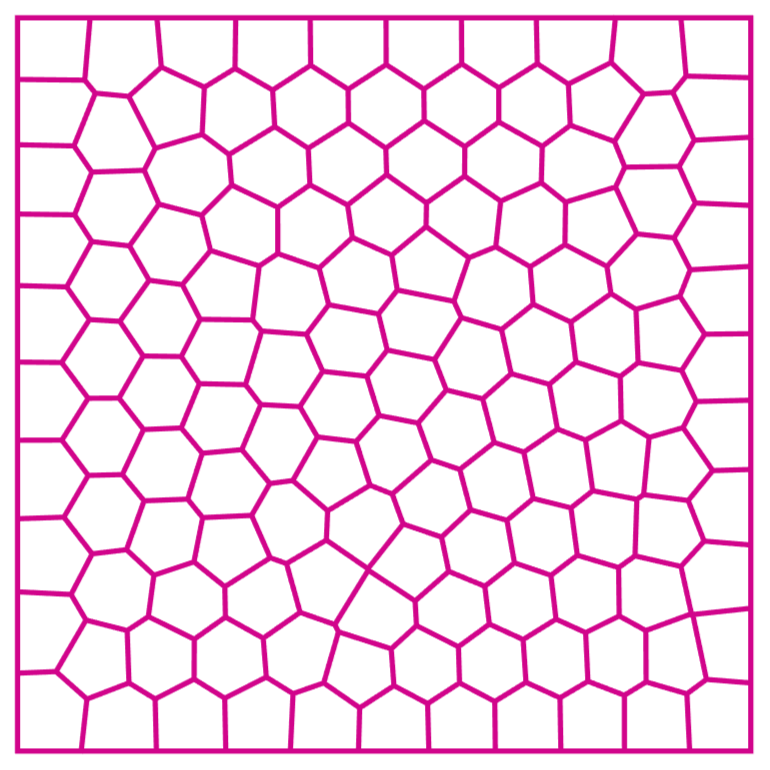} &
    \includegraphics[width=0.22\textwidth]{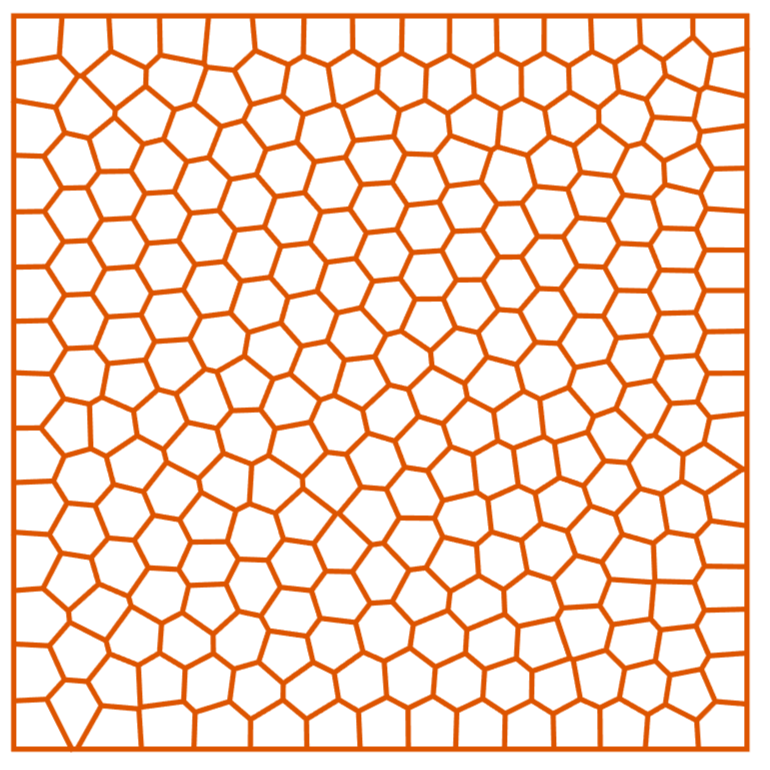} \\

    \raisebox{4.5\height}{\textbf{Level 1}} &
    \includegraphics[width=0.22\textwidth]{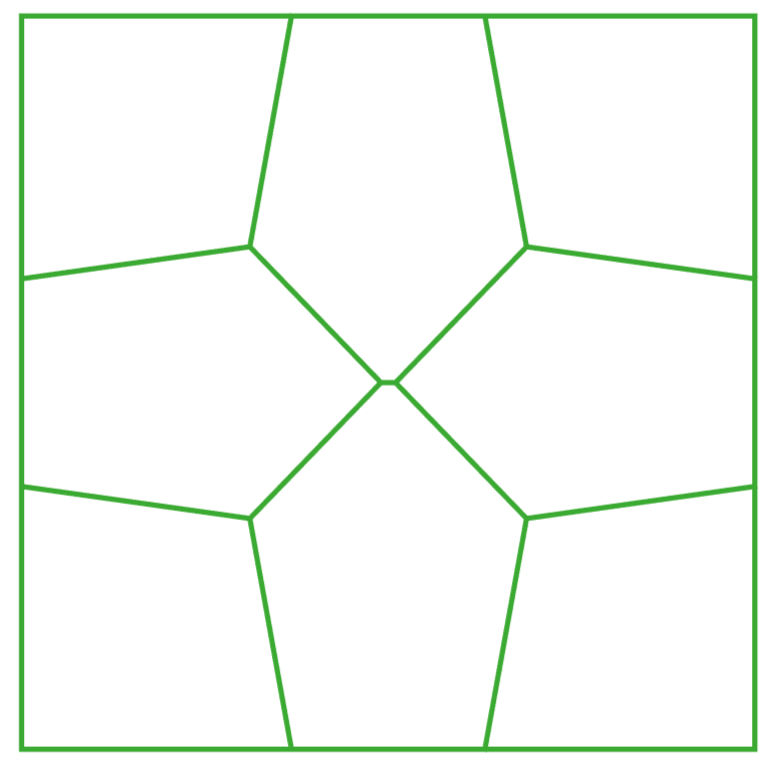} &
    \includegraphics[width=0.22\textwidth]{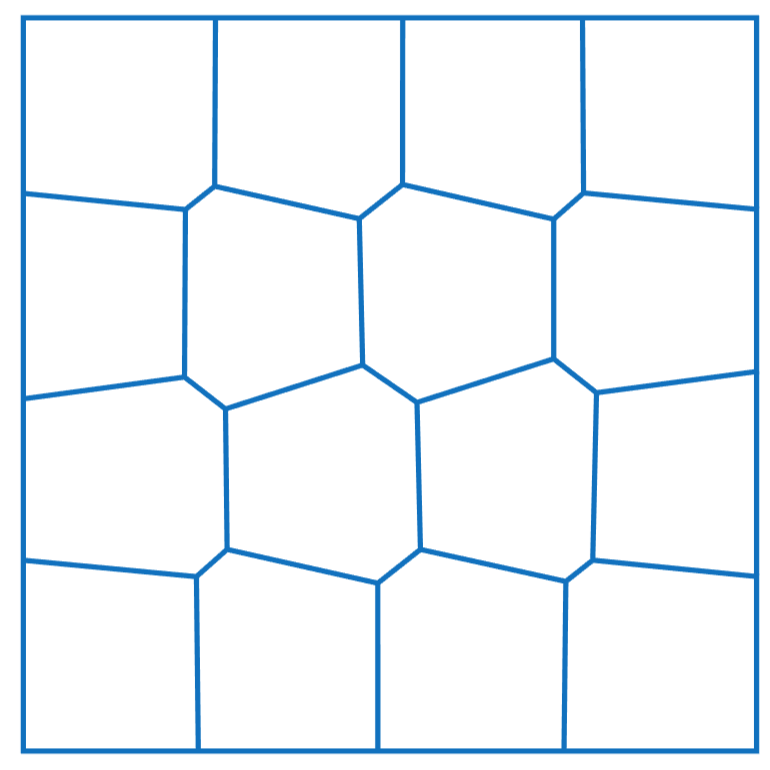} &
    \includegraphics[width=0.22\textwidth]{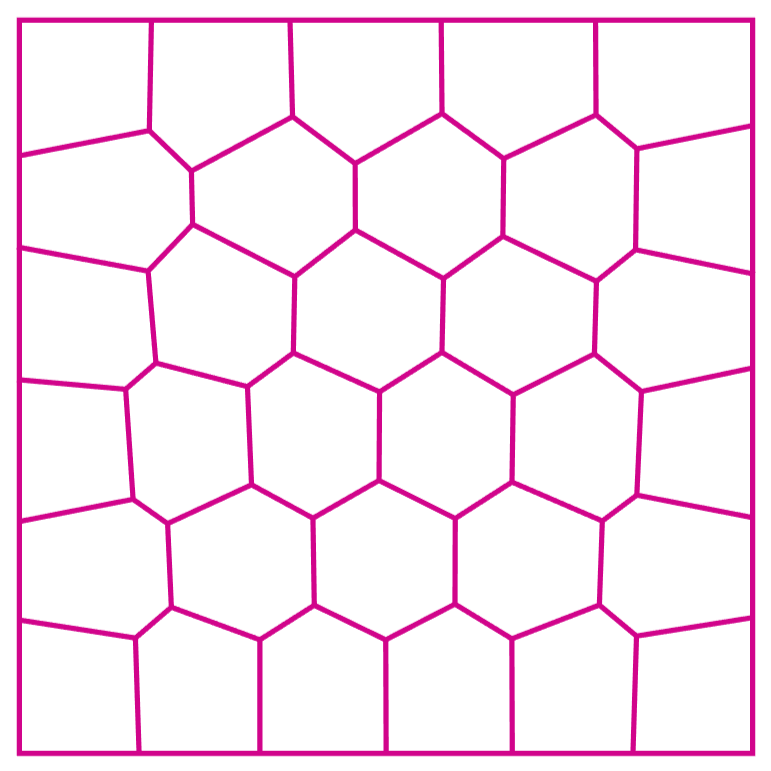} &
    \includegraphics[width=0.22\textwidth]{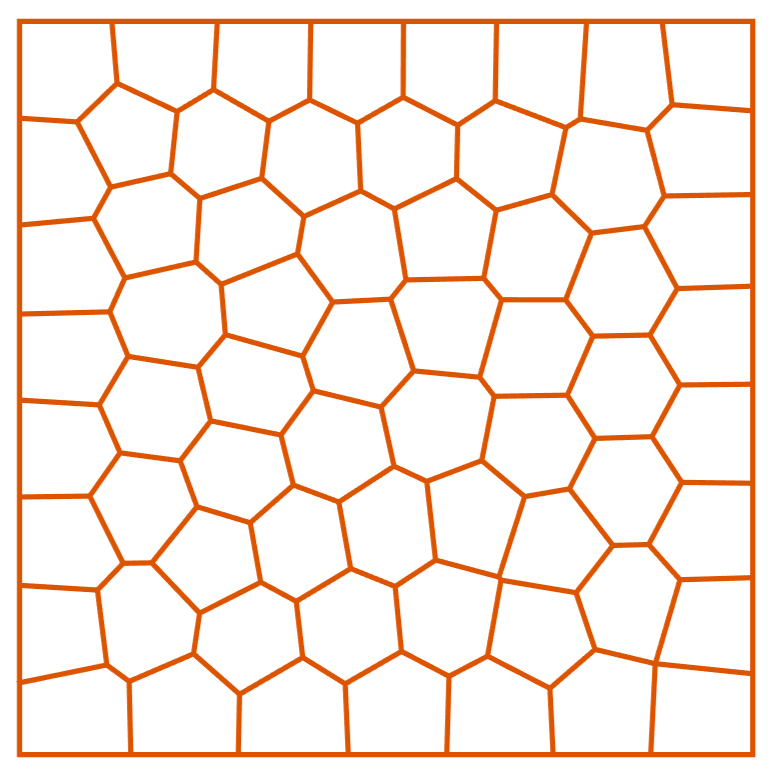} \\

\end{tabular}}

\caption{Hierarchies of non-nested polygonal grids used for numerical simulations of Section~\ref{2dPolyDG}. Finest grid of Set 1 with 512 elements; finest grid of Set 2 with 1024 elements; finest grid of Set 3 with 2048 elements; finest grid of Set 4 with 4096 elements.}
\label{fig:Hierarchy}
\end{figure}

\begin{table}[t]
\centering
\footnotesize
  \caption{$\kappa(A^*)$ (left) and $\kappa_{\mathrm{eff}}(A^*(I-\pi_{A^*}(\mathbf{S})))$ (right) as functions of $h$ and $\Delta t$.}
\label{tab:h_dt}
\renewcommand{\arraystretch}{1.0} 
\setlength{\tabcolsep}{6pt}      
\begin{tabular}{c|cccc|cccc}

\hline
\textbf{$\Delta t$} & \textbf{$h_1$} & \textbf{$h_2$} & \textbf{$h_3$} & \textbf{$h_4$} & \textbf{$h_1$} & \textbf{$h_2$} & \textbf{$h_3$} & \textbf{$h_4$} \\
\hline
$10^{-2}$ &  1.7$\cdot 10^{6}$ & 3.6$\cdot 10^{6}$ & 7.4$\cdot 10^{6}$ & 1.5$\cdot 10^{7}$ & 4.4$\cdot 10^{4}$ & 9.1$\cdot 10^{4}$ & 1.9$\cdot 10^{5}$ & 4.0$\cdot 10^{5}$  \\
$10^{-3}$ &  1.4$\cdot 10^{6}$ & 3.0$\cdot 10^{6}$ & 6.2$\cdot 10^{6}$ & 1.3$\cdot 10^{7}$ & 1.2$\cdot 10^{4}$ & 1.5$\cdot 10^{4}$ & 3.3$\cdot 10^{4}$ & 5.4$\cdot 10^{4}$ \\
$10^{-4}$ &  \cellcolor{lightgray} 1.4$\cdot 10^{6}$ & \cellcolor{lightgray} 2.8$\cdot 10^{6}$ &  5.9$\cdot 10^{6}$ &  1.2$\cdot 10^{7}$ & \cellcolor{lightgray} 3.5$\cdot 10^{3}$ & \cellcolor{lightgray} 4.8$\cdot 10^{3}$ & 9.2$\cdot 10^{3}$ & 1.7$\cdot 10^{4}$ \\
$10^{-5}$ &  \cellcolor{lightgray} 1.4$\cdot 10^{6}$ & \cellcolor{lightgray} 2.8$\cdot 10^{6}$ & \cellcolor{lightgray} 5.8$\cdot 10^{6}$ & \cellcolor{lightgray} 1.2$\cdot 10^{7}$ & \cellcolor{lightgray} 8.3$\cdot 10^{2}$ & \cellcolor{lightgray} 1.3$\cdot 10^{3}$ & \cellcolor{lightgray} 2.4$\cdot 10^{3}$ & \cellcolor{lightgray} 4.2$\cdot 10^{3}$ \\
$10^{-6}$ & \cellcolor{lightgray} 1.7$\cdot 10^{6}$ & \cellcolor{lightgray} 3.2$\cdot 10^{6}$ &  \cellcolor{lightgray} 6.1$\cdot 10^{6}$ & \cellcolor{lightgray} 1.2$\cdot 10^{7}$  & \cellcolor{lightgray} 1.9$\cdot 10^{2}$ & \cellcolor{lightgray} 3.1$\cdot 10^{2}$ & \cellcolor{lightgray} 6.4$\cdot 10^{2}$ & \cellcolor{lightgray} 1.2$\cdot 10^{3}$ \\
$10^{-7}$ &  6.5$\cdot 10^{6}$ & 7.4$\cdot 10^{6}$ & \cellcolor{lightgray} 1.0$\cdot 10^{7}$ & \cellcolor{lightgray} 1.6$\cdot 10^{7}$ & 8.0$\cdot 10^{1}$ & 9.3$\cdot 10^{1}$ & \cellcolor{lightgray} 1.4$\cdot 10^{2}$ & \cellcolor{lightgray} 2.5$\cdot 10^{2}$ \\
$10^{-8}$ & 6.4$\cdot 10^{7}$ & 6.9$\cdot 10^{7}$ & 7.7$\cdot 10^{7}$ & 7.4$\cdot 10^{7}$ & 8.1$\cdot 10^{1}$ & 9.0$\cdot 10^{1}$ & 1.2$\cdot 10^{2}$ & 1.3$\cdot 10^{2}$ \\
$10^{-9}$ & 6.4$\cdot 10^{8}$ & 6.9$\cdot 10^{8}$ & 7.7$\cdot 10^{8}$ & 7.4$\cdot 10^{8}$  & 8.1$\cdot 10^{1}$ & 9.0$\cdot 10^{1}$ & 1.2$\cdot 10^{2}$ & 1.3$\cdot 10^{2}$ \\
$10^{-10}$ & 6.4$\cdot 10^{9}$ & 6.9$\cdot 10^{9}$ & 7.7$\cdot 10^{9}$ & 7.4$\cdot 10^{9}$ & 8.1$\cdot 10^{1}$ & 9.0$\cdot 10^{1}$ & 1.2$\cdot 10^{2}$ & 1.3$\cdot 10^{2}$ \\
\hline
\end{tabular}

\end{table}
Table~\ref{tab:DCGsolvers} reports the iteration counts needed by the CG and the deflated CG algorithms to reduce the relative residual below a tolerance tol $= 10^{-8}$, for different time steps 
$\Delta t$ and mesh sizes $h$. \textcolor{black}{In the deflated CG strategy, the inner system is solved exactly by a direct solver.} For each method, the table reports the number of iterations required to reduce the  relative residual below $10^{-8}$, considering only the solution of the linear system (\ref{pb_study}) arising at a single time step. The iteration counts are obtained as the average over 10 independent tests, in which the initial condition is perturbed using randomly generated data. In particular, Table~\ref{tab:DCGsolvers} shows that the CG iteration counts increase with mesh refinement ($h \to 0$) and exhibit a growth proportional to $\sqrt{\Delta t}$ as $\Delta t \to 0$. In contrast, the deflated CG method yields dramatically fewer iterations as $\Delta t \to 0$, confirming the effectiveness of the deflation strategy. 

\begin{table}[t]
\centering
\footnotesize
 \caption{Iteration counts for CG (left) and deflated CG (right) as functions of $h$ and $\Delta t$, with $\text{tol}=10^{-8}$.}
\label{tab:DCGsolvers}
\renewcommand{\arraystretch}{1.0}
\setlength{\tabcolsep}{6pt}

\centering
\begin{tabular}{c|cccc|cccc}
\hline

$\Delta t$ & $h_1$ & $h_2$ & $h_3$ & $h_4$ & $h_1$ & $h_2$ & $h_3$ & $h_4$ \\
\hline

$10^{-2}$    & 3766 & 5255 & 7538 & 10631 & 2528 & 3463 & 5003 & 6846  \\
$10^{-3}$   & 3110 &  4417 & 6231 & 8080 & 1057 &  1430 & 1989 & 2709 \\
$10^{-4}$  & \cellcolor{lightgray} 2936 & \cellcolor{lightgray} 4110 &  5807 &  8037 & \cellcolor{lightgray} 465 & \cellcolor{lightgray} 609 &  836 &  1091 \\
$10^{-5}$ & \cellcolor{lightgray} 2940 & \cellcolor{lightgray} 4090 & \cellcolor{lightgray} 5741 & \cellcolor{lightgray} 7941 & \cellcolor{lightgray} 208 & \cellcolor{lightgray} 277 & \cellcolor{lightgray} 374 & \cellcolor{lightgray} 489 \\
$10^{-6}$ & \cellcolor{lightgray} 3185 & \cellcolor{lightgray} 4247 & \cellcolor{lightgray} 5845 & \cellcolor{lightgray} 8021 & \cellcolor{lightgray} 102 & \cellcolor{lightgray} 133 & \cellcolor{lightgray} 187 & \cellcolor{lightgray} 243 \\
$10^{-7}$ & 6012 & 5845 & \cellcolor{lightgray} 6745 & \cellcolor{lightgray} 8161 & 69 & 75 & \cellcolor{lightgray} 89 & \cellcolor{lightgray} 113  \\
$10^{-8}$ & 14910 & 15581 & 16618 & 16315 & 70 & 75 & 83 & 84 \\
\hline
\end{tabular}

\end{table}
Table~\ref{tab:DCGsolvers_h_innerp3} shows the results obtained with the deflated CG method, where the 
inner system is now solved by a multigrid W-cycle approach with RAS as smoother. The multigrid hierarchy consists of four levels (as shown in Figure~\ref{fig:Hierarchy}), and we 
apply $m = 5$ pre- and post-smoothing steps. In particular, we report the number of outer iterations of the deflated 
CG method (Out), the total number of multigrid iterations 
required to solve the inner problem (Tot), and the average number of multigrid inner iterations performed per outer iteration (Inn). In the first block row of Table~\ref{tab:DCGsolvers_h_innerp3}, we present the results obtained by setting the tolerance of the inner 
system equal to $0.01$ times the outer tolerance, i.e., as defined in (\ref{fixed_tol}) with $c_F = 0.01$. Then, we replace the fixed tolerance used for the inner multigrid solver with an 
adaptive tolerance. More precisely, at each outer iteration the tolerance for 
the inner system is set as a function of the outer residual, following the 
criterion~(\ref{tol}) with a user-defined parameter $c_A>0$. In our tests, we explore different values of $c_A$ in order to assess their impact on the behavior of the method. Smaller values of $c_A$ enforce a tighter accuracy 
for the inner solver, resulting in more robust convergence of the overall method, 
although at the cost of an increased number of inner iterations. On the other hand, 
larger values of $c_A$ lead to a faster inner solver and fewer iterations, but also 
to a higher incidence of convergence failures in the outer deflated CG method (indicated by ``--'' in Table~\ref{tab:DCGsolvers_h_innerp3}).  
For comparison, the second block row of Table~\ref{tab:DCGsolvers_h_innerp3} presents the behavior for $c_A = 0.02$, while in the third block row we report the results for $c_A = 0.01$, and in the fourth block row those corresponding to $c_A = 0.005$. \textcolor{black}{When convergence is achieved, the adaptive tolerance strategy systematically outperforms the fixed tolerance approach in terms of computational cost, as clearly shown by the total number of multigrid iterations (Tot) reported in the corresponding columns.}

\begin{table}[htbp]
\centering
\footnotesize
 \caption{Iteration counts for deflated CG for different values of $c_F$ in (\ref{fixed_tol}) and $c_A$ in (\ref{tol}).}
\label{tab:DCGsolvers_h_innerp3}
\renewcommand{\arraystretch}{1.0}
\setlength{\tabcolsep}{4pt}

\begin{tabular}{c|c|ccc|ccc|ccc|ccc}
\hline
$c_F \ \text{or} \ c_A$ & $\Delta t$ 
& \multicolumn{3}{c|}{$h_1$} 
& \multicolumn{3}{c|}{$h_2$} 
& \multicolumn{3}{c|}{$h_3$} 
& \multicolumn{3}{c}{$h_4$}\\
\hline
& & Out & Tot & Inn 
& Out & Tot & Inn 
& Out & Tot & Inn 
& Out & Tot & Inn \\
\hline

\multirow{5}{*}{$c_F = 0.01$} &
$10^{-4}$ 
 & \cellcolor{lightgray} 467 & \cellcolor{lightgray}6779 & \cellcolor{lightgray} 15
 & \cellcolor{lightgray} 623 & \cellcolor{lightgray} 9426 & \cellcolor{lightgray} 15
 & 700 & 11100 & 16
 & 1102 & 20764 & 19 \\ &
$10^{-5}$ 
 & \cellcolor{lightgray} 212 & \cellcolor{lightgray} 3066 & \cellcolor{lightgray} 14 
 & \cellcolor{lightgray} 278 & \cellcolor{lightgray} 4232 & \cellcolor{lightgray} 15 
 & \cellcolor{lightgray} 376 & \cellcolor{lightgray} 5790 & \cellcolor{lightgray} 15 
 & \cellcolor{lightgray} 489 & \cellcolor{lightgray} 9115 & \cellcolor{lightgray} 19  \\ &
$10^{-6}$ 
 & \cellcolor{lightgray} 102 & \cellcolor{lightgray} 1512 & \cellcolor{lightgray} 15
 & \cellcolor{lightgray} 135 & \cellcolor{lightgray} 2039 & \cellcolor{lightgray} 15
 & \cellcolor{lightgray} 190 & \cellcolor{lightgray} 2899 & \cellcolor{lightgray} 15
 & \cellcolor{lightgray} 238 & \cellcolor{lightgray} 4435 & \cellcolor{lightgray} 19 \\ &
$10^{-7}$ 
 & 69 & 1006 & 15
 & 75 & 1143 & 15
 & \cellcolor{lightgray} 89 & \cellcolor{lightgray} 1363 & \cellcolor{lightgray} 15
 & \cellcolor{lightgray} 114 & \cellcolor{lightgray} 2017 & \cellcolor{lightgray} 18  \\ &
$10^{-8}$ 
 & 70 & 1023 & 15
 & 74 & 1118 & 15
 & 83 & 1282 & 15
 & 83 & 1524 & 18 \\
\hline

\multirow{5}{*}{$c_A = 0.02$} &
$10^{-4}$ 
 & \cellcolor{lightgray} 479 & \cellcolor{lightgray} 2174 & \cellcolor{lightgray} 5
 & \cellcolor{lightgray} 638 & \cellcolor{lightgray} 3132 & \cellcolor{lightgray} 5
 & 860 & 4196 & 5
 & 1127 & 6136 & 5  \\ &
$10^{-5}$ 
 & \cellcolor{lightgray} -- & \cellcolor{lightgray} --  & \cellcolor{lightgray} -- 
 & \cellcolor{lightgray} -- & \cellcolor{lightgray} -- & \cellcolor{lightgray} -- 
 & \cellcolor{lightgray} -- & \cellcolor{lightgray} -- & \cellcolor{lightgray} -- 
 & \cellcolor{lightgray} 503 & \cellcolor{lightgray} 2074 & \cellcolor{lightgray} 4 \\ &
$10^{-6}$ 
 & \cellcolor{lightgray} -- & \cellcolor{lightgray} -- & \cellcolor{lightgray} --
 & \cellcolor{lightgray} -- & \cellcolor{lightgray} -- & \cellcolor{lightgray} --
 & \cellcolor{lightgray} -- & \cellcolor{lightgray} -- & \cellcolor{lightgray} --
 & \cellcolor{lightgray} -- & \cellcolor{lightgray} -- & \cellcolor{lightgray} -- \\ &
$10^{-7}$ 
 & 70 & 203
 & 3 & 103
 & 243 &  2 & \cellcolor{lightgray} --
 & \cellcolor{lightgray} --  & \cellcolor{lightgray} -- &  \cellcolor{lightgray} -- & \cellcolor{lightgray} -- & \cellcolor{lightgray} -- \\ &
$10^{-8}$ 
 & 70 & 213 & 3
 & 75 & 212 & 3
 & 83 & 216 & 3
 & 117 & 253 & 2 \\
\hline

\multirow{5}{*}{$c_A = 0.01$} &
$10^{-4}$ 
 & \cellcolor{lightgray} 476 & \cellcolor{lightgray} 2411 & \cellcolor{lightgray} 5
 & \cellcolor{lightgray} 634 & \cellcolor{lightgray} 3446 & \cellcolor{lightgray} 5
 & 854 & 4635 & 5
 & 1118 & 6819 & 6  \\ &
$10^{-5}$ 
 & \cellcolor{lightgray} 217 & \cellcolor{lightgray} 888  & \cellcolor{lightgray} 4 
 & \cellcolor{lightgray} 319 & \cellcolor{lightgray} 1264 & \cellcolor{lightgray} 4 
 & \cellcolor{lightgray} 384 & \cellcolor{lightgray} 1616 & \cellcolor{lightgray} 4 
 & \cellcolor{lightgray} 501 & \cellcolor{lightgray} 2315 & \cellcolor{lightgray} 5 \\ &
$10^{-6}$ 
 & \cellcolor{lightgray} 136 & \cellcolor{lightgray} 406 & \cellcolor{lightgray} 3
 & \cellcolor{lightgray} 245 & \cellcolor{lightgray} 575 & \cellcolor{lightgray} 2
 & \cellcolor{lightgray} -- & \cellcolor{lightgray} -- & \cellcolor{lightgray} --
 & \cellcolor{lightgray} 283 & \cellcolor{lightgray} 900 & \cellcolor{lightgray} 3 \\ &
$10^{-7}$ 
 & 69 & 227
 & 3 & 103
 & 266 &  3 & \cellcolor{lightgray} 162
 & \cellcolor{lightgray} 336  & \cellcolor{lightgray} 2 &  \cellcolor{lightgray} 174 & \cellcolor{lightgray} 397 & \cellcolor{lightgray} 2 \\ &
$10^{-8}$ 
 & 70 & 237 & 3
 & 74 & 232 & 3
 & 83 & 243 & 3
 & 85 & 244 & 3 \\
\hline

\multirow{5}{*}{$c_A = 0.005$} &
$10^{-4}$ 
 & \cellcolor{lightgray} 476 & \cellcolor{lightgray} 2635 & \cellcolor{lightgray} 6
 & \cellcolor{lightgray} 629 & \cellcolor{lightgray} 3742 & \cellcolor{lightgray} 6
 & 854 & 5162 & 6
 & 1114 & 7512 & 7  \\ &
$10^{-5}$ 
 & \cellcolor{lightgray} 215 & \cellcolor{lightgray} 971  & \cellcolor{lightgray} 5
 & \cellcolor{lightgray} 303 & \cellcolor{lightgray} 1360 & \cellcolor{lightgray} 4 
 & \cellcolor{lightgray} 382 & \cellcolor{lightgray} 1784 & \cellcolor{lightgray} 5 
 & \cellcolor{lightgray} 499 & \cellcolor{lightgray} 2555 & \cellcolor{lightgray} 5 \\ &
$10^{-6}$ 
 & \cellcolor{lightgray} 134 & \cellcolor{lightgray} 438 & \cellcolor{lightgray} 3
 & \cellcolor{lightgray} 171 & \cellcolor{lightgray} 548 & \cellcolor{lightgray} 3
 & \cellcolor{lightgray} 296 & \cellcolor{lightgray} 791 & \cellcolor{lightgray} 3
 & \cellcolor{lightgray} 276 & \cellcolor{lightgray} 988 & \cellcolor{lightgray} 4 \\ &
$10^{-7}$ 
 & 69 & 249
 & 4 & 78
 & 264 &  3 & \cellcolor{lightgray} 90
 & \cellcolor{lightgray} 296  & \cellcolor{lightgray} 3 &  \cellcolor{lightgray} 148 & \cellcolor{lightgray} 405 & \cellcolor{lightgray} 3 \\ &
$10^{-8}$ 
 & 70 & 261 & 4
 & 74 & 263 & 4
 & 83 & 268 & 3
 & 84 & 269 & 3 \\
\hline

\end{tabular}

\end{table}
When using an adaptive tolerance for the inner multigrid solver, certain choices of the user-defined parameter $c_A$ may lead to situations where the deflated CG method fails to converge. 
To address this issue, we employ the untruncated version of flexible CG, i.e. FCG($\infty$), for the outer solver. 
 To verify the effectiveness of the FCG approach, we replicate the tests reported in Table~\ref{tab:DCGsolvers_h_innerp3}, this time employing the FCG($\infty$) variant for the outer solver. 
The numerical results presented in Table~\ref{tab:FCGsolvers_h_inner_c6p3} show that, with FCG($\infty$), all previously critical cases now converge successfully. Moreover, convergence is preserved in all other non-critical cases, indicating that the use of FCG enhances robustness without compromising the performance of the method in well-behaved scenarios.

\begin{table}[htbp]
\centering
\footnotesize
 \caption{Iteration counts for FCG($\infty$) for different values of $c_F$ in (\ref{fixed_tol}) and $c_A$ in (\ref{tol}).}
\label{tab:FCGsolvers_h_inner_c6p3}
\renewcommand{\arraystretch}{1.0}
\setlength{\tabcolsep}{4pt}

\begin{tabular}{c|c|ccc|ccc|ccc|ccc}
\hline
$c_F \ \text{or} \ c_A$ & $\Delta t$ 
& \multicolumn{3}{c|}{$h_1$} 
& \multicolumn{3}{c|}{$h_2$} 
& \multicolumn{3}{c|}{$h_3$} 
& \multicolumn{3}{c}{$h_4$}\\
\hline
& & Out & Tot & Inn 
& Out & Tot & Inn 
& Out & Tot & Inn 
& Out & Tot & Inn \\
\hline

\multirow{5}{*}{$c_F = 0.01$} &
$10^{-4}$ 
 & \cellcolor{lightgray} 467 & \cellcolor{lightgray}6779 & \cellcolor{lightgray} 15
 & \cellcolor{lightgray} 623 & \cellcolor{lightgray} 9426 & \cellcolor{lightgray} 15
 & 700 & 11100 & 16
 & 1102 & 20764 & 19 \\ &
$10^{-5}$ 
 & \cellcolor{lightgray} 212 & \cellcolor{lightgray} 3066 & \cellcolor{lightgray} 14 
 & \cellcolor{lightgray} 278 & \cellcolor{lightgray} 4232 & \cellcolor{lightgray} 15 
 & \cellcolor{lightgray} 376 & \cellcolor{lightgray} 5790 & \cellcolor{lightgray} 15 
 & \cellcolor{lightgray} 489 & \cellcolor{lightgray} 9115 & \cellcolor{lightgray} 19  \\ &
$10^{-6}$ 
 & \cellcolor{lightgray} 102 & \cellcolor{lightgray} 1512 & \cellcolor{lightgray} 15
 & \cellcolor{lightgray} 135 & \cellcolor{lightgray} 2039 & \cellcolor{lightgray} 15
 & \cellcolor{lightgray} 190 & \cellcolor{lightgray} 2899 & \cellcolor{lightgray} 15
 & \cellcolor{lightgray} 238 & \cellcolor{lightgray} 4435 & \cellcolor{lightgray} 19 \\ &
$10^{-7}$ 
 & 69 & 1006 & 15
 & 75 & 1143 & 15
 & \cellcolor{lightgray} 89 & \cellcolor{lightgray} 1363 & \cellcolor{lightgray} 15
 & \cellcolor{lightgray} 114 & \cellcolor{lightgray} 2017 & \cellcolor{lightgray} 18  \\ &
$10^{-8}$ 
 & 70 & 1023 & 15
 & 74 & 1118 & 15
 & 83 & 1282 & 15
 & 83 & 1524 & 18 \\
\hline

\multirow{5}{*}{$c_A = 0.02$} &
$10^{-4}$ 
 & \cellcolor{lightgray} 425 & \cellcolor{lightgray} 2016 & \cellcolor{lightgray} 5
 & \cellcolor{lightgray} 564 & \cellcolor{lightgray} 2880 & \cellcolor{lightgray} 5
 & 764 & 3846 & 5
 & 1003 & 5629 & 6  \\ &
$10^{-5}$ 
 & \cellcolor{lightgray} 249 & \cellcolor{lightgray} 825  & \cellcolor{lightgray} 3 
 & \cellcolor{lightgray} 355 & \cellcolor{lightgray} 1141 & \cellcolor{lightgray} 3 
 & \cellcolor{lightgray} 500 & \cellcolor{lightgray} 1511 & \cellcolor{lightgray} 3 
 & \cellcolor{lightgray} 467 & \cellcolor{lightgray} 1980 & \cellcolor{lightgray} 4 \\ &
$10^{-6}$ 
 & \cellcolor{lightgray} 132 & \cellcolor{lightgray} 344 & \cellcolor{lightgray} 3
 & \cellcolor{lightgray} 314 & \cellcolor{lightgray} 596 & \cellcolor{lightgray} 2
 & \cellcolor{lightgray} 390 & \cellcolor{lightgray} 744 & \cellcolor{lightgray} 2
 & \cellcolor{lightgray} 428 & \cellcolor{lightgray} 930 & \cellcolor{lightgray} 2 \\ &
$10^{-7}$ 
 & 70 & 203
 & 3 & 102
 & 242 &  2 & \cellcolor{lightgray} 281
 & \cellcolor{lightgray} 427  & \cellcolor{lightgray} 2 &  \cellcolor{lightgray} 276 & \cellcolor{lightgray} 459 & \cellcolor{lightgray} 2 \\ &
$10^{-8}$ 
 & 70 & 213 & 3
 & 75 & 212 & 3
 & 83 & 216 & 3
 & 116 & 252 & 2 \\
\hline

\multirow{5}{*}{$c_A = 0.01$} &
$10^{-4}$ 
 & \cellcolor{lightgray} 423 & \cellcolor{lightgray} 2238 & \cellcolor{lightgray} 5
 & \cellcolor{lightgray} 561 & \cellcolor{lightgray} 3159 & \cellcolor{lightgray} 6
 & 761 & 4270 & 6
 & 996 & 6247 & 6  \\ &
$10^{-5}$ 
 & \cellcolor{lightgray} 206 & \cellcolor{lightgray} 854  & \cellcolor{lightgray} 4 
 & \cellcolor{lightgray} 294 & \cellcolor{lightgray} 1207 & \cellcolor{lightgray} 4 
 & \cellcolor{lightgray} 357 & \cellcolor{lightgray} 1543 & \cellcolor{lightgray} 4 
 & \cellcolor{lightgray} 463 & \cellcolor{lightgray} 2205 & \cellcolor{lightgray} 5 \\ &
$10^{-6}$ 
 & \cellcolor{lightgray} 132 & \cellcolor{lightgray} 400 & \cellcolor{lightgray} 3
 & \cellcolor{lightgray} 233 & \cellcolor{lightgray} 560 & \cellcolor{lightgray} 2
 & \cellcolor{lightgray} 404 & \cellcolor{lightgray} 823 & \cellcolor{lightgray} 2
 & \cellcolor{lightgray} 268 & \cellcolor{lightgray} 875 & \cellcolor{lightgray} 3 \\ &
$10^{-7}$ 
 & 69 & 227
 & 3 & 102
 & 265 &  3 & \cellcolor{lightgray} 159
 & \cellcolor{lightgray} 333  & \cellcolor{lightgray} 2 &  \cellcolor{lightgray} 170 & \cellcolor{lightgray} 393 & \cellcolor{lightgray} 2 \\ &
$10^{-8}$ 
 & 70 & 237 & 3
 & 74 & 232 & 3
 & 83 & 243 & 3
 & 84 & 243 & 3 \\
\hline

\multirow{5}{*}{$c_A = 0.005$} &
$10^{-4}$ 
 & \cellcolor{lightgray} 423 & \cellcolor{lightgray} 2443 & \cellcolor{lightgray} 6
 & \cellcolor{lightgray} 561  & \cellcolor{lightgray} 3440 & \cellcolor{lightgray} 6
 & 759 & 4734 & 6
 & 994 & 6881 & 7  \\ &
$10^{-5}$ 
 & \cellcolor{lightgray} 203 & \cellcolor{lightgray} 941  & \cellcolor{lightgray} 5
 & \cellcolor{lightgray} 280 & \cellcolor{lightgray} 1308 & \cellcolor{lightgray} 5 
 & \cellcolor{lightgray} 355 & \cellcolor{lightgray} 1707 & \cellcolor{lightgray} 5 
 & \cellcolor{lightgray} 462 & \cellcolor{lightgray} 2429 & \cellcolor{lightgray} 5 \\ & 
$10^{-6}$ 
 & \cellcolor{lightgray} 130 & \cellcolor{lightgray} 435 & \cellcolor{lightgray} 3
 & \cellcolor{lightgray} 164 & \cellcolor{lightgray} 538 & \cellcolor{lightgray} 3
 & \cellcolor{lightgray} 280 & \cellcolor{lightgray} 769 & \cellcolor{lightgray} 3
 & \cellcolor{lightgray} 261 & \cellcolor{lightgray} 964 & \cellcolor{lightgray} 4 \\ &
$10^{-7}$ 
 & 69 & 249
 & 4 & 77
 & 263 &  3 & \cellcolor{lightgray} 89
 & \cellcolor{lightgray} 295  & \cellcolor{lightgray} 3 &  \cellcolor{lightgray} 144 & \cellcolor{lightgray} 401 & \cellcolor{lightgray} 3 \\ &
$10^{-8}$ 
 & 70 & 261 & 4
 & 74 & 263 & 4
 & 83 & 268 & 3
 & 84 & 269 & 3 \\
\hline
\end{tabular}

\end{table}
\begin{remark} \label{remark_fcg}
\textcolor{black}{In this work we consider the variant FCG($\infty$), where no truncation of the search directions is applied. If a mixed truncation–restart strategy \cite{Notay} is employed, with $m_0 = 0$ and $m_i = \max\{1, \mathrm{mod}(i, m_{\max}+1)\}$, the performance of the method depends on the user-defined parameters $(c_A, m_{\max})$. In particular, $m_{\max}$ must be tuned with respect to $c_A$, as different combinations may lead to either robust or deteriorated convergence. For this reason, we propose the untruncated FCG($\infty$) approach, which avoids the need to tune the parameters.}  
\end{remark}

\begin{remark}\label{rmk:gmres}
\textcolor{black}{For completeness, we also tested a GMRES variant with adaptive 
inner tolerance, following the same strategy as the FCG($\infty$) 
counterpart. In all tested configurations, GMRES converged 
robustly, without exhibiting the convergence issues observed 
with standard CG under inexact inner solves. GMRES therefore 
represents a viable alternative to FCG($\infty$).}
\end{remark}

\subsection{DG Finite Element discretisation - 3D case} \label{3d_simulations}
In this section, all results are obtained with the FEniCS software. We consider $\Omega = (0,1)^3$ and we set $\boldsymbol{F} = \boldsymbol{0}$ and $\boldsymbol{\sigma}_0 = \boldsymbol{0}$ as initial condition. We impose $\bm \nabla \cdot \bm \sigma = (sin(\pi y)sin(\pi z),0,0)^\top$ on $\Gamma_{\text{in}} = \{0\} \times (0,1) \times (0,1)$, \textcolor{black}{$\bm \sigma \bm n = \bm 0$  on $\Gamma_{\text{out}} =\{1\} \times (0,1) \times (0,1)$, and $\bm \nabla \cdot \bm \sigma = \boldsymbol{0}$ in the remaining part of the boundary}. For the numerical simulations, we set $\mu = 0.5$, the penalty coefficient $\alpha^* = 40$, we fix $p=1$ and $\alpha = \Delta t$ (implicit Euler scheme in time). We consider three nested meshes: mesh~$1$ consisting of 48 tetrahedra ($h_1 \approx 0.8660$), mesh~$2$ made by 384 tetrahedra ($h_2 \approx 0.4330$), and mesh~$3$ with 3072 tetrahedra ($h_3 \approx 0.2165$). These meshes are first employed to compare the performance of CG and deflated CG methods, and are then used to construct the multigrid hierarchy, repeating the tests carried out in Section~\ref{2dPolyDG} for the 2D case.
In particular, in Table~\ref{tab:DCGsolvers_3D} (left) we report the iteration counts needed by the CG and the deflated CG algorithms to reduce the relative residual below $10^{-8}$, for different time steps 
$\Delta t$ and mesh sizes $h$. In this case, in the deflated CG strategy, the inner system is solved exactly by a direct solver. The 3D results exhibit the same trend observed in 2D: CG deteriorates with mesh refinement and as $\Delta t \to 0$, while the deflated CG remains uniformly efficient and robust with respect to $\Delta t$.

\begin{table}[t]
\centering
\footnotesize
\caption{Left table: Iteration counts for CG (left) and deflated CG (right) as functions of $h$ and $\Delta t$, with $\text{tol}=10^{-8}$. Right table: Iteration counts for deflated CG  for mesh $3$, for different values of $c_F$ in (\ref{fixed_tol}) and $c_A$ in (\ref{tol}).}
\label{tab:DCGsolvers_3D}
\renewcommand{\arraystretch}{1.0}
\setlength{\tabcolsep}{2pt} 

\begin{minipage}{0.42\textwidth}
\centering
\begin{tabular}{c|ccc|ccc}
\hline
$\Delta t$ & $h_1$ & $h_2$ & $h_3$ & $h_1$ & $h_2$ & $h_3$ \\
\hline
$1$       & \cellcolor{lightgray} 897 & 1537 & 2947 & \cellcolor{lightgray} 271  & 445 & 823 \\
$0.5$     & \cellcolor{lightgray} 767 & \cellcolor{lightgray} 1311 & \cellcolor{lightgray} 2555 & \cellcolor{lightgray} 242 & \cellcolor{lightgray} 394 & \cellcolor{lightgray} 733 \\
$0.1$     & \cellcolor{lightgray} 477 & \cellcolor{lightgray} 845 & \cellcolor{lightgray} 1646 & \cellcolor{lightgray} 153 & \cellcolor{lightgray} 273 & \cellcolor{lightgray} 483 \\
$10^{-2}$ & 252 & \cellcolor{lightgray}  405 & \cellcolor{lightgray} 769 & 58 & \cellcolor{lightgray}  112 & \cellcolor{lightgray} 219  \\
$10^{-3}$ &  258 &  328 & 523 &  24 &  41 & 75  \\
$10^{-4}$ &  329 & 417 & 578 & 11 & 17 & 29  \\
$10^{-5}$ &  362 & 528 & 855 & 7 & 9 & 13 \\
\hline
\end{tabular}
\end{minipage}
\hfill 
%
\begin{minipage}{0.55\textwidth}
\centering
\begin{tabular}{c|ccc|ccc|ccc}
\hline
$\Delta t$ 
& \multicolumn{3}{c|}{$c_F = 0.01$} 
& \multicolumn{3}{c|}{$c_A = 0.02$} 
& \multicolumn{3}{c}{$c_A = 0.005$}\\
\hline
& Out & Tot & Inn
& Out & Tot & Inn
& Out & Tot & Inn \\
\hline
$0.5$ 
  &  972 & 31625
 &  33 & 1081 &  12687
 & 12 & 1017 & 14494 & 14 \\ 
$0.1$ 
  &  522 & 16656
 &  32 & 603 &  7684
 & 13 & 559 & 8741 & 16 \\ 
$10^{-2}$ 
&  219 & 6595
 &  30 & 242 &  2905
 & 12 & 220 & 3344 & 15 \\ 
$10^{-3}$ 
 & 75 & 2148
 & 29 & 75 & 928 &  12
 &  75 &  1081 &  14  \\ 
 $10^{-4}$ 
 & 29 & 824
 & 28 & 29 & 342 &  12
 &  29 &  407 &  14  \\ 
\hline
\end{tabular}
\end{minipage}

\end{table}

%
%
%
%
We now repeat the same set of tests for the deflated CG method, where the inner system is now solved by a multigrid W-cycle approach with RAS as smoother, in the 3D setting. In this case, the finest grid used is the one consisting of 3072 tetrahedral elements (mesh~$3$). The multigrid hierarchy is therefore built on the 
three nested levels (mesh~$1$ $\subset$ mesh~$2 \subset$ mesh~$3$), and we employ $m = 10$ 
pre- and post-smoothing steps.
In Table~\ref{tab:DCGsolvers_3D} (right) we test one value of the fixed parameter $c_F$ and two values of the adaptive parameter $c_A$ to assess their effect on the efficiency and robustness of the method. The observed results are consistent with the 2D case, confirming the same qualitative behavior.

%
%
%
%
In Figure~\ref{fig:plots3d}, we report the pressure and velocity field corresponding to the test case introduced above. The pressure is reconstructed from the numerical pseudo-stress solution through the relation $p = -\frac1d {\rm tr}(\boldsymbol{\sigma})$, while the velocity is recovered by using
$
\boldsymbol{u}(t)=\boldsymbol{u}_0+\int_0^t\big(\boldsymbol{\nabla}\cdot\boldsymbol{\sigma}(s)\big)\,\mathrm{d}s.
$
The results shown in the figure are obtained on the finest mesh of the hierarchy (i.e., mesh~3 consisting of 3072 tetrahedral elements) with a time step $\Delta t=0.5$. To reconstruct the velocity, a single implicit Euler step is performed. The simulation is carried out using the deflated CG method, with the inner linear systems solved by a three-level nested multigrid algorithm and an adaptive inner tolerance (with $c_A=0.02$). The plots confirm that the proposed approach is able to accurately reproduce the expected physics of the system.

\begin{figure}[t]
\centering
\resizebox{0.90\textwidth}{!}{
\begin{tabular}{cc}

\includegraphics[width=0.48\textwidth]{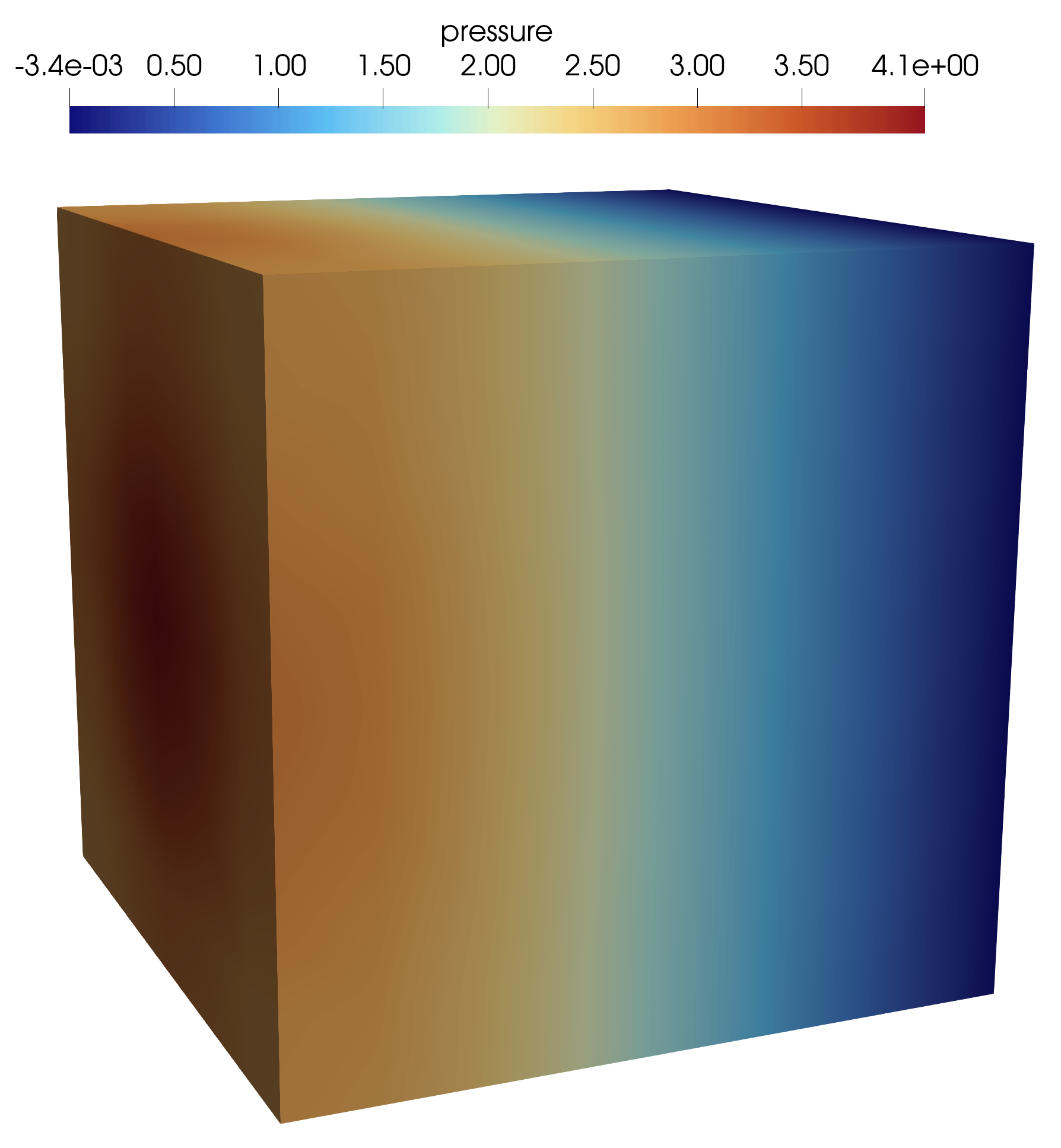} &
\includegraphics[width=0.48\textwidth]{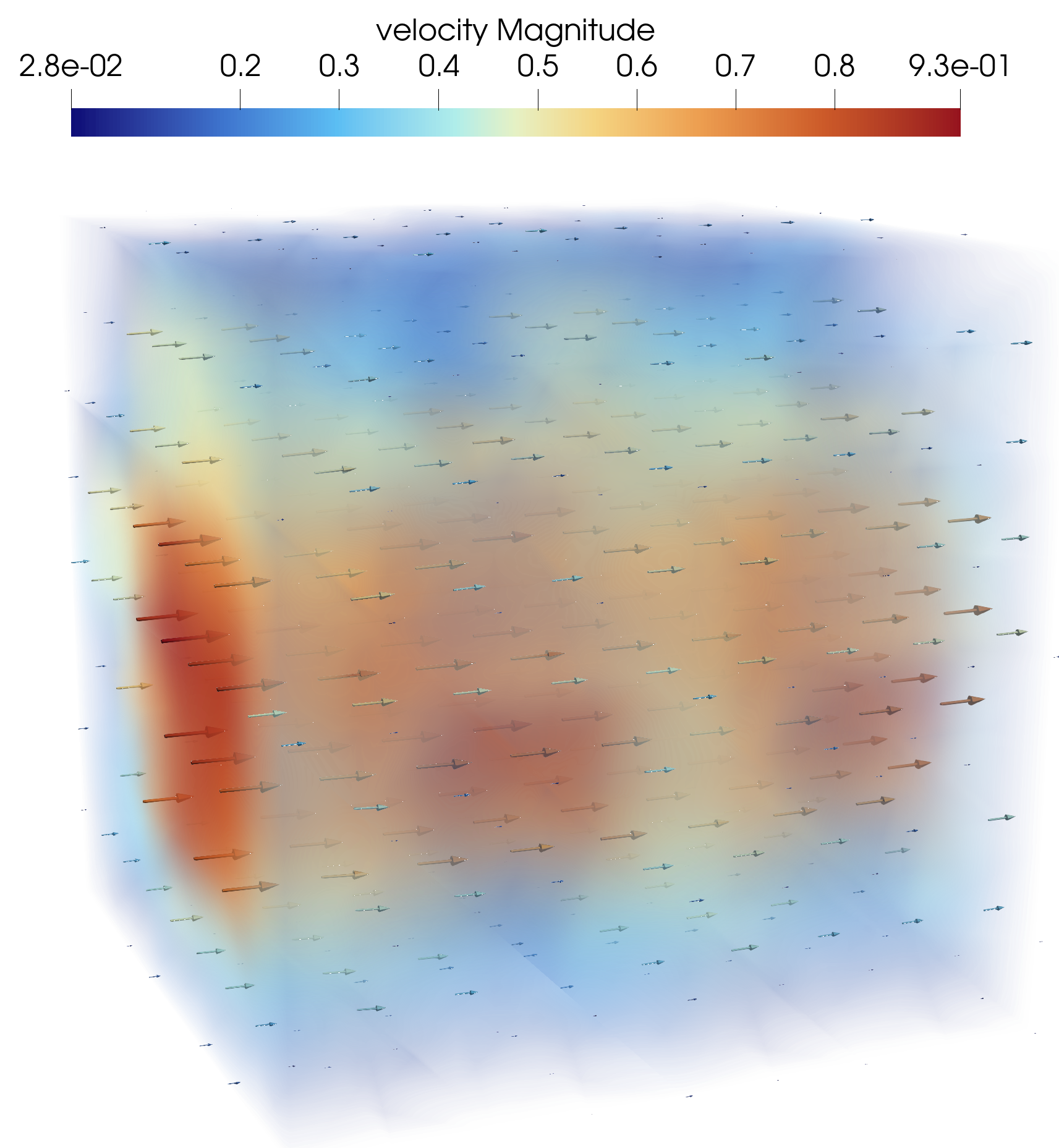} \\


\end{tabular}
}

\caption{Test case of Section~\ref{3d_simulations}. Left: computed pressure field $p_h$ by means of the relation $p_h= -\frac13 \mathrm{tr}(\bm \sigma_h)$. Right: computed velocity field $\bm u_{h}$ by means of the relation $\bm u_h =\bm u_{0,h} + \int_0^T \boldsymbol{\nabla}\cdot\boldsymbol{\sigma}_h(s)\, ds $, with final time $T=0.5$.}
\label{fig:plots3d}
\end{figure}

\section{Conclusions} \label{sec:conclusions}
In this work, we have proposed a numerical solver for the linear system arising from conforming and discontinuous finite element discretisations of the pseudo-stress unsteady Stokes problem. The main difficulty is related to the deterioration of the conditioning as the time step decreases. To address this issue, we introduce a deflation-based strategy that removes the contribution of the kernel components, thereby ensuring robustness with respect to the time step.
The resulting inner system is solved through an efficient multigrid solver. Since the inner system is solved only approximately, we employ a flexible Conjugate Gradient method to guarantee stable convergence of the overall algorithm.
As a result, the proposed approach combines deflation, multigrid acceleration, and flexible Krylov techniques into a scalable solver. The robustness of the method is confirmed by numerical results, which show stable convergence with respect to the time step.
Possible further developments may include the investigation of multigrid methods as preconditioners for the outer Conjugate Gradient solver, in order to achieve robustness also with respect to the mesh size. Another interesting direction is the extension of the proposed solver to other time discretisation schemes, including Runge--Kutta methods, for the pseudo-stress Stokes problem.

\section*{Acknowledgments}
This work is funded by the European Union (ERC SyG, NEMESIS, project number 101115663). Views and opinions expressed are however those of the authors only and do not necessarily reflect those of the European Union or the European Research Council Executive Agency. All authors are members of the Indam GNCS group.

\end{document}